\documentclass[12pt]{article}%
\usepackage{amssymb}
\usepackage{amsfonts}
\usepackage{amsmath}
\usepackage{color}
\usepackage{graphicx}%
\usepackage{geometry}
\usepackage{cite}
\usepackage{hyperref}
\providecommand{\U}[1]{\protect\rule{.1in}{.1in}}
\newtheorem{theorem}{Theorem}

\newtheorem{corollary}[theorem]{Corollary}

\newtheorem{definition}[theorem]{Definition}

\newtheorem{lemma}[theorem]{Lemma}

\newtheorem{remark}[theorem]{Remark}

\newenvironment{proof}[1][Proof]{\noindent\textbf{#1.} }{\ \rule{0.5em}{0.5em}}
\numberwithin{equation}{section}
\begin{document}

\title{Long-time behavior of a nonlocal Cahn-Hilliard equation with reaction 
}
\author{{Annalisa Iuorio\thanks{ Vienna University of Technology, Institute for
Analysis and Scientific Computing, Wiedner Hauptstra{\ss}e 8-10, 1040 Vienna, Austria. $\qquad$ E-mail:
\textit{annalisa.iuorio@tuwien.ac.at}}, Stefano Melchionna\thanks{ University of Vienna, Faculty of Mathematics,
Oskar-Morgenstern-Platz 1, 1090 Vienna, Austria. $\qquad$ E-mail:
\textit{stefano.melchionna@univie.ac.at}} } { }}
\maketitle

\begin{abstract}
In this paper we study the long-time behavior of a nonlocal Cahn-Hilliard
system with singular potential, degenerate mobility, and a reaction term. In particular, we prove the existence
of a global attractor with finite fractal dimension, the existence of an exponential attractor, and convergence to
equilibria for two physically relevant classes of reaction terms.

\end{abstract}




\section{Introduction}

In this paper we aim to study the long-time
behavior of a nonlocal Cahn-Hilliard system with singular potential, 
degenerate mobility, and a reaction term given by
\begin{subequations}
\label{nonlocCH}
\begin{align}
\frac{\partial u}{\partial t}-\nabla\cdot(\mu(u) \nabla v) & =g(u)\text{ in
}Q\text{,}\label{nonloc CH 1}\\
v & =f^{\prime}(u)+w\text{ in }Q\text{,}\label{nonloc CH 2}\\
w(x,t) & =\int_{\Omega}K(\left\vert  x-y\right\vert  )(1-2u(y,t))\mathrm{d}
y \ \text{ for }(x,t)\in Q\text{,}\label{nonloc CH 3}\\
n\cdot\mu\nabla v & =0~\ \text{on }\Gamma,\label{bc}\\
u(x,0) & =u_{0}(x)\text{, }x\in\Omega\text{,} \label{IC}\\%
f(u) &= u\log u+(1-u)\log(1-u)\text{,}\label{def f}\\
\mu(u) &= \frac{1}{f''(u)} = u(1-u)\text{.}\label{def mu}%
\end{align}
\end{subequations}
where the spatial domain $\Omega\subset\mathbb{R}^{d}$, $d \leq 3$, is assumed to be bounded and with Lipschitz boundary $\partial\Omega$,
$Q=\Omega\times(0,+\infty)$, $\Gamma=\partial\Omega\times(0,+\infty)$ and $n$
denotes the outward normal unit vector on $\partial\Omega$.

The Cahn-Hilliard equation \cite{Ca, CH} arises in the context of phase transitions. These are
defined as the changes of a system from one regime or state to another exhibiting
very different properties. Two types of models have been traditionally adopted in the literature to
describe their occurrence: \textit{sharp-interface} and \textit{phase-field} models.
The first ones describe the interface between two components in a
mixture (e.g., liquid/solid or two chemical species) as a $(d-1)$-dimensional
hypersurface.
\\
\noindent
In phase-field models the sharp interface is replaced by a thin transition region in
which a mixture of the two components is present. The main unknown is here
represented by a real valued function $u$ which describes the local concentration of
one of the two components. In contrast with sharp-interface models,
$u$ is allowed to vary continuously in the interval between the pure concentration values
(say $0$ and $1$). This approach allows us to avoid enforcing complicated boundary conditions
across the interface as well as being concerned with regularity issues.

The Cahn-Hilliard equation has been originally introduced as a phase-field model 
to study the phenomena of \textit{spinodal decomposition} (loss of
mixture homogeneity and formation at a fine scale of pure phase regions) and
\textit{coarsening} (aggregation of pure phase regions into larger domains) in a binary alloy. The original
model is a gradient-flow (in the $H^{-1}$ metric) of the free energy functional given by~\cite{CH} 
\begin{equation}
\label{eq:enfunCH}
E_{CH}(u)=\int_{\Omega}\frac{\tau^{2}}{2}|\nabla u|^{2}+F(u),%
\end{equation}
where $\tau$ is a small positive parameter related to the transition region thickness and $F$ is a double well
potential attaining its two global minima in correspondence of the pure phases (represented here by the values $0$
and $1$). The related evolution problem is then given by the $H^{-1}$-gradient flow associated with the energy
functional
\eqref{eq:enfunCH}%
\begin{align*}
\frac{\partial u}{\partial t}+\nabla\cdot J_{CH} & =0\text{,}\\
J_{CH} & =-\mu(u)\nabla v_{CH}\text{,}\\
v_{CH} & =\frac{\delta E_{CH}(u)}{\delta u}=F^{\prime}(u)-\tau^{2}\Delta
u\text{,}%
\end{align*}
where the function $\mu(u)$ is known as mobility.

\noindent
Though quite successful and largely studied, the Cahn-Hilliard model cannot be rigorously derived as a macroscopic limit of microscopic models of interacting particles. Considering the hydrodynamic limit of such a microscopic model, Giacomin and Lebowitz \cite{GL} derived a nonlocal energy functional of the form
\begin{equation}
\label{eq:enfunGL}
E(u)=\int_{\Omega}\int_{\Omega}\tilde{K}(x-y)(u(x)-u(y))^{2}\mathrm{d}%
x\mathrm{d}y+\eta\int_{\Omega}F(u(x))\mathrm{d}x,
\end{equation}
where $\tilde{K}$ is a positive and symmetric convolution kernel and $\eta$ is a positive parameter. The associated evolution problem is a
\textit{nonlocal} variant of the Cahn-Hilliard system%
\begin{align*}
\frac{\partial u}{\partial t}+\nabla\cdot J & =0\text{,}\\
J & =-\mu(u)\nabla v\text{,}\\
v & =\frac{\delta E(u)}{\delta u}=\eta F^{\prime}(u)+\bar{k}(x)u-\tilde
{K}\ast u\text{,}%
\end{align*}
where $\bar{k}(x):=\int_{\Omega}\tilde{K}(x-y)\mathrm{d}y$. We note that (under suitable choices of $F$, $\eta$, and
$\tilde{K}$ \cite{GZ}) the system can be rewritten in the form%
\begin{subequations}
\label{nl ch no reaction}
\begin{align}
\frac{\partial u}{\partial t}-\nabla\cdot(\mu\nabla v) & =0\text{,} \\
v & =f^{\prime}(u)+w\text{,}\\
w(x,t) & =\int_{\Omega}K(x-y)(1-2u(y,t))\mathrm{d}y.
\end{align}
\end{subequations}
Here $K$ is again a symmetric positive convolution kernel and the potential $f$ is convex. 

Local and nonlocal systems (along with several variants) have been widely studied in the last years from different
prospectives, e.g., well posedness \cite{BH, EG, GZ, DPG, MR}, qualitative properties \cite{LP2, CMZ},
numerical aspects~\cite{GRW,GWW}, applications \cite{BEG, BO, KS}, long-time behavior
\cite{LP1, AW, GG, Mi, CMZ14}, and asymptotics \cite{AB, MM, GNS, LR, Le}, just to
mention a few.
At first glance, the local and nonlocal equations seem to differ considerably:
the first one is a fourth-order PDE, while the nonlocal one is an integro-differential
parabolic equation. Nevertheless, many fundamental features are shared by the two systems,
such as the gradient flow structure, the lack of comparison principles, and
the separation of the solution from the pure phases \cite{LP2, CMZ}.
Moreover, formal computations show that, with suitably
choices of convolution kernels, the nonlocal energy functional $E$ converges to the local one $E_{CH}$\footnote{The
local term $\tau |\nabla u|^2$ can be obtained as the formal limit of the corresponding nonlocal term
with kernel $K(x,y)=m^{d+2}J(|m(x-y)|^2)$
as $m \to \infty$, where $J$ is a nonnegative function with compact support.}.
In addition, the two energy functionals admit the same $\Gamma$-limit for vanishing interface thickness \cite{MM, AB} (see also \cite{Le, GNS} for the
sharp interface limit of the local Cahn-Hilliard equation).

Including a reaction term can be crucially significant in applications;
some examples are image processing \cite{BEG, CFM}, chemistry \cite{BO,Mi},
and biological models \cite{KS,CMZ14,Fa}. We note that its introduction in the equations 
sensibly influences their properties. It is well known that solutions of
the Cahn-Hilliard system without reaction (i.e. $g=0$) entail conservation of the total mass
($\frac{\mathrm{d}}{\mathrm{d}t}\int_{\Omega}u=0$) and separate uniformly from
the pure phases, i.e., $k\leq u(t)\leq1-k$ for all $t\geq t_{0}$ for some
positive $k=k(t_{0})\ $independent of $t$. These properties are no longer
satisfied in the case $g\neq0$, as a trivial consequence of the relation
$\frac{\mathrm{d}}{\mathrm{d}t}\int_{\Omega}u=\int_{\Omega}g(u)$. Moreover,
the system with reaction loses its gradient-flow structure, or, in other
words, does not admit (to the best of the authors' knowledge) a Lyapunov
functional. Finally, uniqueness of stationary solutions is lost.
Consequently, the study of the asymptotic behavior of solutions
to \eqref{nonlocCH} seems to be of particular interest.

Although the non-local Cahn-Hilliard equation is often simplified in the literature by considering a regular (polynomial)
potential and a nondegenerate (i.e. bounded away from zero) mobility,
here, as in~\cite{MR, GG2}, we consider
a \textit{singular (logarithmic)
potential} and a \textit{degenerate mobility}.
This setting is indeed more physically feasible relevant for it allows us to prove that the solution $u$
remains bounded between the pure phases $0$ and $1$, i.e.
$0\leq u(t)\leq1$ for all $t\geq0$. Due to the lack of a comparison principle for both local and
nonlocal Cahn-Hilliard equations, this property cannot be proved (and in
general does not hold true) in the case of regular potential and nondegenerate
mobility.

As a consequence of~\eqref{def mu}, system~\eqref{nonloc CH 1}-\eqref{nonloc CH 2} can be rewritten as a
nonlocal perturbation of a parabolic equation
\begin{equation}\tag{\ref{nonloc CH 1}*}
\label{nonloc CH 1.2}
\frac{\partial u}{\partial t}-\Delta u-\nabla\cdot(\mu\nabla w)= g(u) \text{.}%
\end{equation}
\noindent
Existence, uniqueness, regularity, and separation from pure phases have been already proved in \cite{MR}. Here, we
address questions concerning the asymptotic behavior of solutions. In particular, we first prove the existence of the
global attractor with finite fractal dimension. Let us recall that
the global attractor $\mathcal{A}$ describes the asymptotic dynamics of the system. In simple words,
every trajectory will be eventually close to the set $\mathcal{A}$. Knowing that $\mathcal{A}$ has finite dimension
is thus relevant as it implies that a finite number of variables can describe the long-time dynamics of the system
with good approximation. This is fundamental, e.g., in numerical analysis.\\
\noindent
Secondly, we show the existence of an exponential attractor $\mathcal{A}^{\prime }$. 
Even though uniqueness and invariance properties cannot be proved in this case, we gain an important information
about convergence, which was lacking for the global attractor. \\
\noindent
Finally, we prove existence of equilibria, i.e., weak solutions to the equation
\[
-\nabla\cdot(\mu\nabla v)=g(u), \qquad 0\leq u\leq1.
\]
Most importantly, we get convergence results to steady states for specific
classes of reaction terms $g$. 
In the literature there exist several results about convergence to
equilibria for the reaction-free Cahn-Hilliard equation. Most of
these results (cf. \cite{AW} for the local case and \cite{LP1, LP2}
for the nonlocal case) make use of the gradient-flow structure of
the equation. More precisely, the existence of a Lyapunov functional
is crucial to prove convergence to equilibria, for instance by using
a technique based on the {\L}ojasiewicz-Simon theorem \cite{FIP}. In our
case, however, the existence of a Lyapunov-type functional is
unknown. Therefore, traditionally used techniques cannot be applied
in our setting.
We obtain convergence results for reaction terms $g$ with a definite sign
(namely, $g\geq0$ or $g\leq0$) or strictly monotone decreasing. In
the first case, we observe that the total mass is monotone as a
function of time. Combined with the boundedness of $u$, this proves
the convergence of the solution to one of the pure phases.
In the second case we exploit a linearization technique around the
equilibrium point. We emphasize that these choices include a
wide spectrum of reaction terms particularly relevant in
applications.

\textbf{Plan of the paper.} The paper is structured as follows. 
In Section 2 we present the assumptions, revise some preliminary results,
and state the main theorems. Section
3 is devoted to a detailed proof of the existence and finite fractal
dimension of the global attractor. Existence of an exponential
attractor is illustrated in Section 4. Finally, existence and convergence
to steady states are explored in Section 5, where the two
different reaction-term scenarios are analyzed separately.


\section{Assumptions and main results}


\subsection{Assumptions}

We assume that the given functions $K$, $u_0$ and $g$ fulfill the following conditions:

\begin{description}
\item[(K)] The convolution kernel $K:%
\mathbb{R}
^{d}\rightarrow%
\mathbb{R}
$ satisfies
\begin{equation}
K(x)=K(-x)\text{ for a.a. }x\in%
\mathbb{R}^d \tag{K1}
\text{,} \label{K1}%
\end{equation}%
\begin{equation}
\sup_{x\in\Omega}\int_{\Omega}|K(x-y)|\mathrm{d}y<+\infty\text{,} \label{K2} \tag{K2}
\end{equation}%
\begin{equation}
\forall p\in\lbrack1,+\infty]~\exists r_{p}>0\text{ such that } \Vert K\ast
\rho\Vert _{W^{1,p}(\Omega)}\leq r_{p}\Vert \rho\Vert _{L^{p}(\Omega)}~\forall\rho\in
L^{p}(\Omega)\text{,} \label{K3} \tag{K3}
\end{equation}%
\begin{equation}
\exists C>0\text{ such that }\Vert K\ast\rho\Vert _{H^{2}(\Omega)}\leq C\Vert \rho
\Vert _{H^{1}(\Omega)}~\forall\rho\in H^{1}(\Omega)\text{;} \label{K4} \tag{K4}
\end{equation}

\item[(U0)] The initial datum $u_{0}\in L^{\infty}(\Omega)$ is such that
\begin{equation}
 0 \leq u_{0}(x) \leq 1 \ \ \text{for a.a.}\ x \in \Omega \text{,} \tag{U01}
\end{equation}
\begin{equation}
 0<\bar{u}_{0}<1 \text{,}  \tag{U02}
\end{equation}
where we use the notation $\bar{z}= \frac{1}{|\Omega|} \int_\Omega z$;

\item[(G)] The reaction term $g:(x,s)\in\Omega\times\lbrack0,1]\longmapsto
g(x,s)\in%
\mathbb{R}
$ is such that
\begin{equation}
g\  \text{is measurable in $x$ and uniformly $C^{1,1}$ in $s$} \text{,} \label{G1} \tag{G1}
\end{equation}
\begin{equation}
g(x,0) \geq 0 \geq g(x,1)\ \ \text{for a.a.}\ x \in \Omega. \label{G2} \tag{G2}%
\end{equation}
\end{description}
Finally, we recall that $f(u) = u\log u+(1-u)\log(1-u)$ and  $\mu(u) = u(1-u)$. In particular, $\mu(u)= 1/f''(u)$.

Examples of convolution kernels $K$ satisfying conditions (K)
are given by 
\begin{itemize}
\item Gaussian kernels: $K(\left\vert  x\right\vert 
)=C\exp(-\left\vert  x\right\vert  ^{2}/\lambda)$,
\item Mollifiers:
$ K(x)=\left\{
\begin{array}
[c]{ll}%
C\exp(\frac{-h^{2}}{h^{2}-|x|^{2}}) & \text{ if }|x|<h,\\
0 & \text{ if }|x|\geq h,
\end{array}
\right.$
\item Newton potentials:
$
\left\{
\begin{array}
[c]{lll}%
K(\left\vert  x\right\vert  )=k_{d}\left\vert  x\right\vert  ^{2-d} & \text{for }
& d>2,\\
K(\left\vert  x\right\vert  )=-k_{2}\ln\left\vert  x\right\vert   & \text{for} &
d=2,
\end{array}
\right.
$
\end{itemize}
where $h,\lambda,k_{d}>0$.

Hypotheses (G) cover a wide range of reaction terms occurring in applications; some relevant examples are given by
\begin{align}
g(x,u)  &  =\alpha(x)u(1-u)\text{,}\label{bio react term}\\
g(x,u)  &  =\beta(x)(h(x)-u)\text{,}\label{chb}\\
g(x,u)  &  =-\sigma(x)u\text{,} \label{cho}%
\end{align}
where $\alpha,\beta,h,\sigma\in L^{\infty}(\Omega)$ are positive functions and
$h\leq1$ a.e. in $\Omega$. The function defined by (\ref{bio react term}) is a space-dependent
logistic map largely employed in population dynamics (cf. \cite{KS}). Equation~\eqref{nonlocCH} together with \eqref{chb} is known as the Cahn-Hilliard-Bertozzi equation and is used in image inpainting  \cite{BEG}. Finally, equation~\eqref{nonlocCH}
with \eqref{cho}, i.e. the so-called Cahn-Hilliard-Oono equation, was introduced in the context of diblock copolymers \cite{BO, Mi}.
%
\subsection{Preliminaries}

In our work we deal with weak solutions to equation~\eqref{nonlocCH}. This notion is made precise by the following definition:

\begin{definition}
\label{def sol} Let assumptions \emph{(K), (U0), (G)} be satisfied. $u$ is a \emph{weak solution} to \emph{\eqref{nonloc CH 1}%
-\eqref{IC}} on $[0,+\infty)$ if
\begin{equation}
 u\in L_{\text{loc}}^2(0,+\infty; H^1(\Omega)) \cap H_{\text{loc}}^1(0,+\infty;\left(  H^{1}(\Omega)\right)  ^{\ast}),
\end{equation}
\begin{equation}
 0  \leq u \leq 1 \text{ a.e. in } Q,
\end{equation}
\begin{equation}
 w  = K\ast(1-2u) \text{ a.e in } Q,
\end{equation}
\begin{equation}
 w \in C \left( \left[
  0,+\infty \right); W^{1,\infty}(\Omega) \right)
\end{equation}
\begin{equation}
 u(0)  = u_{0} \text{ a.e. in } \Omega \text{,}
\end{equation}
and the following equality is satisfied a.e. in $(0,+\infty)$ and for every $\psi\in
H^{1}(\Omega)$%
\[
\left\langle \dot{u},\psi\right\rangle _{H^{1}(\Omega)}+(\mu(u)\nabla
w,\nabla\psi)_{L^{2}(\Omega)}+(\nabla u,\nabla\psi)_{L^{2}(\Omega)}%
=(g(u),\psi)_{L^{2}(\Omega)}.
\]
\end{definition}

Several important properties of the solutions such as existence, uniqueness, continuous dependence on the initial datum, and
separation from pure phases have been proved in \cite{MR}:

\begin{theorem}[Existence, uniqueness and separation properties,~\cite{MR}]
\label{existence and separation}Let assumptions \emph{(K), (U0), (G)} be satisfied. Then:
\begin{enumerate}
\item[(i)] There exists a unique weak solution $u$ to \emph{\eqref{nonlocCH}};
\item[(ii)] If $u_{1},u_{2}$ are two solutions of \emph{\eqref{nonlocCH}}
corresponding to initial data $u_{01},u_{02}$ respectively, the continuous dependence estimate%
\[
\Vert u_{1}-u_{2}\Vert _{L^{\infty}(0,t;L^{2}(\Omega))}\leq\mathrm{e}^{Ct}\Vert u_{01}%
-u_{02}\Vert _{L^{2}(\Omega)}%
\]
holds true for some constant $C=C(K,\Omega)>0$;
\item[(iii)] For every $0<t_{0}<T$, $u\in L^{\infty}(t_{0},T;H^{2}(\Omega))$;
\item[(iv)] For every $0<t_{0}<T$, there exist $k_{1}=k_{1}(\bar
{u}_{0},t_{0},T)>0$ and $k_{2}=k_{2}(\bar{u}_{0},t_{0},T)>0$ such that
\[
k_{1}\leq u(x,t) \leq1-k_{2}\quad\text{ for a.a.}~x\in\Omega\text{ and }%
t_0 \leq t \leq T;%
\]
\item[(v)] If $g\geq0$ a.e.~in $\Omega\times\lbrack0,1],$ $k_{1}$ is
independent of $T$. Similarly, if $g\leq0$ a.e. in $\Omega\times\lbrack0,1],$
$k_{2}$ is independent of $T.$
\end{enumerate}
\end{theorem}

Theorem \ref{existence and separation} allows us to define a dynamical
system $(X,S(t))$, where the state space
\[
X=\{u\in L^{2}(\Omega):0\leq u\leq1 \text{ a.e. in } \Omega\}
\]
is equipped with the $L^2(\Omega)$-topology
and $S(t)$ is the (strongly continuous) solution operator associated with~\eqref{nonlocCH}
at time $t$, i.e., $S(t)u_{0}=u(t)$,
where $u$ is the solution to~\eqref{nonlocCH}.

\subsection{Main results}

Our first result shows the existence of the global attractor for problem~\eqref{nonlocCH}. The notion
of global attractor (see Def.~10.4,~\cite{Ro}) is one of the most discussed and widely used tools
to investigate the long-time behavior of an evolution equation.
The global attractor $\mathcal{A}$ is the maximal compact totally-invariant subset of $X$, i.e., such that
\[
 S(t)\mathcal{A} = \mathcal{A} \hspace{1cm} \text{ for all } t \geq 0,
\]
or equivalently the minimal set attracting all bounded sets $B$, i.e. such that
\begin{equation}
 \text{dist}_\mathcal{H}(S(t)B,\mathcal{A}) \to 0 \hspace{1cm} \text{ as } t \to \infty. \label{attr}
\end{equation}
Here, $\text{dist}_\mathcal{H}$ is the so-called Hausdorff semidistance defined by 
$$
\text{dist}_\mathcal{H}(\mathcal{B},\mathcal{C})=\sup_{b \in \mathcal{B}} \inf_{c \in \mathcal{C}} \Vert b-c \Vert_{L^2(\Omega)},
$$
for all $\mathcal{B}$ and $\mathcal{C}$ nonempty subsets of X.
As a consequence of~\eqref{attr}, one can say that the long-time dynamics of the system is well described by the
dynamics close to the attractor. 
Although the existence of the global attractor is already a valuable information, often the set $\mathcal{A}$
is not explicitly described and (though compact) might be large and difficult to characterize. This motivates us in proving that 
the global attractor has finite fractal dimension.\\
\noindent
The notion of fractal dimension is a generalization of the standard notion of Hausdorff dimension. Following
(Def.~13.1,~\cite{Ro}) we define
\begin{definition}[Fractal dimension] \label{deffracdim}
If $\bar{Y}$ is compact in a Hilbert space $B$, the \emph{fractal dimension} of $Y$
with respect to the topology of $B$ is defined as
\begin{equation}
\label{fracdim}
d_{\mathrm{frac}}^B(Y) = \underset{\varepsilon \to 0}{\operatorname{limsup}} \frac{\log N(Y, \varepsilon)}{\log(1/\varepsilon)},
\end{equation}
where $N(Y,\varepsilon)$ denotes the minimum number of $B$-balls of radius $\varepsilon$ needed to cover $Y$.
\end{definition}
The theorem hence reads as follows:

\begin{theorem}[Global Attractor] \label{thm:ga}
 Let assumptions \emph{(K), (U0), (G)} be satisfied.
 Then, the global attractor for the dynamical system $(X,S(t))$
 \begin{enumerate}
  \item[(i)] exists and is connected;
  \item[(ii)] has finite fractal dimension.
 \end{enumerate}
\end{theorem}
We prove Theorem~\ref{thm:ga} in Section~\ref{sec:ga}.

\noindent
The results concerning the global attractor, however, do not give us any indication
about the rate of convergence of the solutions. Hence we investigate the existence of
a finite-dimensional exponential attractor $\mathcal{A}^{\prime }$ for system~\eqref{nonlocCH} with respect to the
$L^2(\Omega)$ metric.
\begin{definition} \label{def:ea}
A compact set $\mathcal{A}^{\prime }$ is called \emph{exponential attractor} with respect to the
$L^2(\Omega)$ metric if there exist some positive constants $c, C>0$ independent of $t$ and $u_0$
such that
\[
d_{L^{2}(\Omega )}(S(t)u_{0},\mathcal{A}) = \sup_{a\in \mathcal{A}}\left\Vert
S(t)u_{0}-a\right\Vert _{L^{2}(\Omega )} \leq
C \mathrm{e}^{-tc}\quad\text{ for
all }u_{0}\in X \text{ and } t \geq 0.
\]%
\end{definition}
\noindent
Our main result concerning exponential convergence is then summarized in the following theorem:
\begin{theorem}[Exponential attractor] \label{thm:ea}
 Let assumptions \emph{(K), (U0), (G)} be satisfied.
 Then, there exists an exponential attractor with respect to the $L^2(\Omega)$-metric
 for the dynamical system $(X,S(t))$.
\end{theorem}
We remark that analogous results have been already proved in the reaction-free case in~\cite{GG2} (see also~\cite{FG,
FGR} for a system coupled with Navier-Stokes equations).
Theorems \ref{thm:ga} and \ref{thm:ea} can be interpreted as an extension of these results in several directions. First,
we allow
for a nontrivial reaction term. Secondly, the dynamical system considered in~\cite{GG2} has as a state space
the set of functions $0 \leq u \leq 1$ whose spatial average $\bar{u}$ is bounded away from the pure phases
$0$ and $1$ by an arbitrarily small but fixed constant. The reason for this restriction is that uniform (in time)
separation from pure phases ensures better estimates, in particular $L^\infty$ estimates on the chemical
potential, that are used to prove the exponential convergence to the attractor. Moreover, the finite dimension of the
attractor is proved in~\cite{GG2} with respect to a metric weaker than the $L^2$ metric. 
The restriction on $\bar{u}$ is of course unnatural in the case $g \neq 0$ for the quantity $\bar{u}$ is no longer 
preserved in time. This forced us to find a different way to derive estimates. By applying a generalization of the Gronwall
Lemma (see Lemma~\ref{ugl} in the Appendix) we can bound $\nabla u$ in $L^\infty(t_0, \infty, L^2(\Omega))$ (cf. estimate~\eqref{gradient
unif bound}). This strong estimate provides sufficient regularity to obtain the stronger statement of
Theorem~\ref{thm:ga}.

\noindent
We now focus our attention on the steady states of system~\eqref{nonlocCH}. 
These correspond to solutions of the equation
\begin{equation}
\label{equil}
 -\nabla \cdot (\mu\nabla v)=g(u).
\end{equation}
In the reaction-free case (i.e, $g=0$), the uniform-in-time
separation properties proved in~\cite{LP2} allow us to explicitly obtain equilibrium
triples $(u^{\ast},v^{\ast},w^{\ast})$, where
\[
\begin{aligned}
 u^{\ast} &= \frac{1}{1+\exp(w^{\ast}-v^{\ast})}, \hspace{.5cm} \int_{\Omega}u^{\ast}  &  =\int_{\Omega}u_{0},\\
 v^{\ast}  &= const, \\
 w^{\ast}  &= K\ast(1-2u^{\ast}).
\end{aligned}
\]
On the other hand, in the case $g\neq0$ the separation
properties are not uniform in time as already observed in~\cite{MR}.
In particular, the chemical potential $V$ is not well defined.
This suggests to rewrite equation~\eqref{equil}
as
\begin{subequations}
\label{eqeq}
\begin{align}
-\Delta u-\nabla\cdot(\mu\nabla w)  &  =g(u)\text{ in }\Omega\text{,}%
\label{equilibrium equation}\\
n\cdot(\mu\nabla w+\nabla u)  &  =0\text{ on }\partial\Omega\text{,}\\
w  &  =K\ast(1-2u)\text{ in }\Omega. \label{equilibrium equation 3}%
\end{align}
\end{subequations}
\noindent
More precisely, we give the following definition.
\begin{definition}
Let \emph{(K), (U0), (G)} be satisfied. Then, $u$ is
called \emph{equilibrium point} for \emph{\eqref{nonlocCH}} or
\emph{weak solution} to \emph{\eqref{eqeq}} if $u\in H^{1}(\Omega)$, $0\leq u\leq1$ a.e.~in
$\Omega$, and it satisfies%
\begin{equation} \label{eq:equil}
\left(  \nabla u+\mu\nabla w,\nabla\psi\right)  _{L^{2}(\Omega)}=\left(
g(u),\psi\right)  _{L^{2}\left(  \Omega\right)  } \hspace{.5cm} \forall\psi\in H^{1}%
(\Omega)\text{,}%
\end{equation}
where $w$ satisfies \emph{(\ref{equilibrium equation 3})} pointwise a.e. in
$\Omega$.
\end{definition}
\noindent
Differently from the reaction-free case
(i.e., $g=0$), existence of stationary solutions for~\eqref{nonlocCH} is not trivial and has to be proved
directly. Moreover, uniqueness is false in the general case (see Remark~\ref{rem:nue}).
We recall that, as we do not know any Lyapunov functional for system~\eqref{nonlocCH}, convergence to equilibria is
hard to obtain. However, restricting the class of reaction terms $g$ we can prove
convergence to equilibria. More precisely, we consider functions $g$ which either have a sign (see Case $1$ in Theorem~\ref{thm:eq}) or are monotone decreasing (see Case $2$).\\
\noindent
Our result reads as follows.

\begin{theorem}[Equilibria] \label{thm:eq}
 Let assumptions \emph{(K), (U0), (G)} hold.
 Then, we have the following:
 \begin{description}
 \item[Existence.] There exists at least one equilibrium point $u\in H^{1}(\Omega)$ such that $0\leq
u\leq1$ a.e.~in $\Omega$, i.e., a solution of~\eqref{eq:equil}.
 \item[Convergence - Case 1.] If in addition $g$ is such that
 \begin{subequations}
\label{pr9hypA}  
\begin{align}
 & g(x,s)\geq0\ \quad \text{for a.a.~} x \in \Omega \text{ and } \forall s \in [0,1], \label{hA1} \tag{A1}\\
 & \forall \kappa>0\ \exists\hspace{.03cm}m_\kappa>0\ \text{ s.t. } \ g(x,s)\geq-m_\kappa(s-1) \quad 
 \forall s\in\lbrack\kappa,1], \text{ for a.a.~} x \in \Omega, \label{hA2} \tag{A2}
\end{align}
\end{subequations}
then, $u(t)\rightarrow1$ exponentially fast in $L^{2}(\Omega)$ for every $u_{0} \in X$ if the measure of 
the set $\{x \in X \text{ s.t. } g(x,0) > 0\}$ is positive and for
every $u_{0}\neq0$ if $g(\cdot,0) \equiv 0$.
Similarly, if $g$ is such that
\begin{subequations}
\label{pr9hypB}  
 \begin{align}
 & g(x,s)\leq0\ \quad \text{for a.a.~} x \in \Omega \text{ and } \forall s \in [0,1], \tag{B1}\\
 & \forall \kappa>0\ \exists\hspace{.03cm}m_\kappa>0\ \text{ s.t. } \ g(x,s)\leq-m_\kappa s 
 \quad \forall s\in \lbrack0,1-\kappa], \text{ for a.a.~} x \in \Omega, \tag{B2}
\end{align}
\end{subequations}
then $u(t)\rightarrow0$ exponentially fast in $L^{2}(\Omega)$ for every $u_{0} \in X$ if the measure of 
the set $\{x \in X \text{ s.t. } g(x,1) < 0\}$ is positive and for
every $u_{0}\neq1$ if $g(\cdot,1) \equiv 0$.
\item[Convergence - Case 2.] If in addition $g$ is uniformly differentiable in the second variable and
\begin{equation} \label{hC}
\partial_s g(x,s)\leq-\lambda<0 \text{ for a sufficiently large } \lambda=\lambda(\Omega,K) \text{ for
a.a.}~(x,s)\in\Omega\times\lbrack0,1], 
\end{equation}
then there exists a unique equilibrium $u^\ast$ and $u(t)\rightarrow u^{\ast}$ exponentially fast in $L^{2}(\Omega)$
for $t\rightarrow+\infty$.
 \end{description}
\end{theorem}

Many physically relevant reaction terms \cite{KS,BO,BEG} are included in these scenarios, such as the ones illustrated in~\eqref{bio react term}-\eqref{cho}. 

\begin{corollary}[Cahn-Hilliard with logistic reaction term] \label{cor:log}
Let $g$ be as in~\eqref{bio react term} with $\alpha(x)\geq\alpha_{0}>0$ a.e. in $\Omega$.
Then, the solution $u$ to~\eqref{nonlocCH} 
converges to $1$ exponentially fast in $L^{2}(\Omega)$ for every initial datum $u_0$ not
identically $0$. Moreover, $u=0$ is an equilibrium point.
\end{corollary}

\begin{corollary}[Cahn-Hilliard-Bertozzi] \label{cor:CHB}
Let $g$ be as in~\eqref{chb} with $\beta(x)\geq\beta_{0}>0$ a.e.~in $\Omega$ for $\beta_{0}=\beta_{0}(\Omega,K)$ sufficiently
large.
Then, the solution $u$ converges
to the unique stationary point $u^{\ast}$ exponentially fast in $L^{2}(\Omega)$.
\end{corollary}

\begin{corollary}[Cahn-Hilliard-Oono] \label{cor:CHO}
Let $g$ be as in~\eqref{cho} with $\sigma(x)\geq\sigma_{0}>0$ a.e.~in $\Omega$.
Then, the solution $u$ converges to $0$ exponentially fast in $L^{2}(\Omega)$.
\end{corollary}


\section{The global attractor: Proof of Theorem~\ref{thm:ga}} \label{sec:ga}

\subsection{Part (i): Existence}

The existence of the global attractor is based on proving that $S(t)$ is dissipative
and possesses a compact absorbing set (cf. Theorem $10.5$ in \cite{Ro}).
The first property holds since $X$ is $\Vert \cdot\Vert _{L^{2}(\Omega)}$-bounded. In order to show
the existence of a compact absorbing set, we simply derive uniform $L^2$-estimates on the
gradient of $\nabla u$. To this aim, we start by testing equation
\eqref{nonloc CH 1} with $u$ and estimate, by using boundedness of $u$ and $g$
and assumption \eqref{def mu} on $\mu$,
\[
\frac{\mathrm{d}}{\mathrm{d}t} \frac12 \Vert u\Vert _{L^{2}(\Omega)}^{2}+\Vert \nabla
u\Vert _{L^{2}(\Omega)}^{2}+\left(  \mu\nabla w,\nabla \right)  _{L^{2}(\Omega)} \leq \int_\Omega g(u)u\leq
C.%
\]
Thus, integrating over $[t,t+1]$ and using $0 \leq u \leq 1$, we get
\[
\int_{t}^{t+1}\left(  \Vert \nabla u\Vert _{L^{2}(\Omega)}^{2}+\left(  \mu\nabla w,\nabla u\right)  _{L^{2}(\Omega)} \right)  \leq C+ \frac12 \Vert u(t)\Vert _{L^{2}(\Omega)}^{2} \leq
C+\frac{1}{2}|\Omega|.
\]
Hence%
\[
\sup_{t\in\lbrack0,+\infty)}\int_{t}^{t+1}\left(  \Vert \nabla u\Vert _{L^{2}(\Omega
)}^{2}+\left(  \mu\nabla w,\nabla u \right)_{L^{2}(\Omega)} \right)  <C_1
\]
where $C$ does not depend on $u_{0}$.

\noindent
Testing now equation \eqref{nonloc CH 1} with $\dot{u}$ we get
\begin{equation}
\Vert \dot{u}\Vert _{L^{2}(\Omega)}^{2}+\frac{\mathrm{d}}{\mathrm{d}t}\frac{1}%
{2}\Vert \nabla u\Vert _{L^{2}(\Omega)}^{2}+\left(  \mu\nabla w,\nabla \dot{u} \right)  _{L^{2}(\Omega)}=\int_{\Omega}g\left(  u\right)  \dot{u}\text{.}
\label{est1}%
\end{equation}
Note that $\left(  \mu\nabla w,\nabla \dot{u}\right)  _{L^{2}(\Omega
)}=-\int_{\Omega}\frac{\mathrm{d}}{\mathrm{d}t} \left(  \mu\nabla w\right) \nabla u+\frac{\mathrm{d}%
}{\mathrm{d}t}\int_{\Omega}\mu\nabla w\nabla u$. Using properties of $u$, $\mu$ and $K$, 
we estimate
\[
\left|  \int_{\Omega}\frac{\mathrm{d}}{\mathrm{d}t}\left(  \mu\nabla w\right) \nabla u\right| 
\leq C \Vert \dot{u}\Vert _{L^{2}(\Omega)}\Vert \nabla u\Vert _{L^{2}(\Omega)}\leq C \Vert \nabla
u\Vert _{L^{2}(\Omega)}^{2}+\frac{1}{2}\Vert \dot{u}\Vert _{L^{2}(\Omega)}^{2}.
\]
By substituting into (\ref{est1}) and using Young's inequality, we get
\begin{align}
\frac{1}{2}\Vert \dot{u}\Vert _{L^{2}(\Omega)}^{2}+\frac{\mathrm{d}}{\mathrm{d}%
t}\left(  \frac{1}{2}\Vert \nabla u\Vert _{L^{2}(\Omega)}^{2}+\int_{\Omega}\mu\nabla
w\nabla u\right)   &\leq \int_{\Omega}g\left( u\right)  \dot{u}+C\Vert \nabla
u\Vert _{L^{2}(\Omega)}^{2}\nonumber\\
&\leq C+C\Vert \nabla u\Vert _{L^{2}(\Omega)}^{2}+\frac{1}{4}\Vert \dot{u}%
\Vert _{L^{2}(\Omega)}^{2}\text{.} \label{est2}%
\end{align}
Note that, thanks to~\eqref{def mu}, \eqref{K3}, using H\"older and Young's inequalities,
\begin{align*}
\Vert \nabla u\Vert _{L^{2}(\Omega)}^{2}  &  =\Vert \nabla u\Vert _{L^{2}(\Omega)}^{2}%
+2\int_{\Omega}\mu\nabla w\nabla u-2\int_{\Omega}\mu\nabla w\nabla u\\
&  \leq\Vert \nabla u\Vert _{L^{2}(\Omega)}^{2}+2\int_{\Omega}\mu\nabla w\nabla
u+\frac{1}{2}\Vert \nabla u\Vert _{L^{2}(\Omega)}^{2}+C.
\end{align*}
Thus,
\begin{equation}
\label{ineq1}
\frac{1}{2}\Vert \nabla u\Vert _{L^{2}(\Omega)}^{2}\leq\Vert \nabla u\Vert _{L^{2}(\Omega
)}^{2}+2\int_{\Omega}\mu\nabla w\nabla u+C.
\end{equation}
By substituting into (\ref{est2}), we finally have
\[
\frac{\mathrm{d}}{\mathrm{d}%
t}\left(  \frac{1}{2}\Vert \nabla u\Vert _{L^{2}(\Omega)}^{2}+\int_{\Omega}\mu\nabla
w\nabla u\right)  \leq C_2 \left(  \frac{1}{2}\Vert \nabla u\Vert _{L^{2}(\Omega
)}^{2}+\int_{\Omega}\mu\nabla w\nabla u\right) + C_3.
\]
Applying the Uniform Gronwall Lemma (Lemma~\eqref{ugl}, see Appendix) with
\[
\eta(t)=\frac{1}{2}\Vert \nabla u(t)\Vert _{L^{2}(\Omega)}^{2}+\int_{\Omega}\mu\nabla
w\nabla u,
\]
$r=1$, $t_{0}=0$, $\phi(t)=a_1=C_2$, $\psi(t)=a_2=C_3$ and $a_{3}=C_1$
we get
\[
\frac{1}{2}\Vert \nabla u(t)\Vert _{L^{2}(\Omega)}^{2}+\int_{\Omega}\mu(t)\nabla
w(t)\nabla u(t)\leq C \hspace{.7cm} \forall t \geq 1.
\]
This yields, as a consequence of~\eqref{ineq1}, the following relevant estimate
\begin{align}
\label{gradient unif bound}
\Vert \nabla u(t)\Vert _{L^{2}(\Omega)}^{2}\leq C_4 \hspace{.7cm} \forall t \geq 1,
\end{align}
where $C_3$ does not depend on $u_{0}$. This proves that $B_{1}=\{v\in
H^{1}(\Omega) \subset X:\Vert \nabla v\Vert _{L^{2}\left(  \Omega\right)  }\leq C_3\}$ is an absorbing set for
$(X,S(t))$ which is also compact in $L^{2}(\Omega)$. By virtue of Theorem~$10.5$ of \cite{Ro}, there
exists the global attractor $\mathcal{A}$ for $(X,S(t))$.

\subsection{Part (ii): Finite fractal dimension}

In this section, we show that $d_{\mathrm{frac}}(\mathcal{A})$ is finite. To this aim we adopt a standard strategy presented
e.g. in~\cite{Ro}.
The underlying heuristic idea is the following. Consider an arbitrary infinitesimal $n$-dimensional set $\mathcal{V}$,
(e.g. a cube) and follow its evolution under the action of the semigroup $S(t)$. Suppose that for all $\mathcal{V}$
the $n$-dimensional volume vanishes for large times, i.e., 
\[
 \text{Vol}_n(S(t)\mathcal{V}) \to 0 \quad \text{ as } t \to \infty.
\]
Since the global attractor $\mathcal{A}$ is the union of all the $\omega$-limit sets, this means that $\mathcal{A}$
contains no $n$-dimensional subsets. Thus $d_{\mathrm{frac}}(\mathcal{A}) \leq n$. \\
\noindent
In order to do this in a rigorous way we first need a technical condition ensuring us that the flow is smooth enough
to linearize the equation around a trajectory:

\begin{lemma}[Uniform differentiability]
\label{lemma unif dif}
$S(t)$ is uniformly differentiable, i.e., for all
$u_{0}\in\mathcal{A}$ there exists a linear operator $\Lambda
(t,u_{0}):L^{2}(\Omega)\mathcal{\rightarrow}L^{2}(\Omega)$ such that for all
$t\geq0$,%
\begin{align}
\sup_{v_{0}:0<\left\Vert  u_{0}-v_{0}\right\Vert  _{L^{2}(\Omega)}%
\leq\varepsilon} & \frac{\left\Vert  S(t)v_{0}-S(t)u_{0}-\Lambda(t,u_{0}%
)(v_{0}-u_{0})\right\Vert  _{L^{2}(\Omega)}}{\left\Vert  u_{0}-v_{0}\right\Vert 
_{L^{2}(\Omega)}} \rightarrow0 & \text{ as }\varepsilon\rightarrow
0\label{unif diff 1}, \\
\text{and } \hspace{1cm}\sup_{u_{0}\in\mathcal{A}} & \left\Vert  \Lambda(t,u_{0})\right\Vert  _{L^{2}%
(\Omega)\rightarrow L^{2}(\Omega)} <\infty & \text{ for all }t\geq
0.\label{unif diff 2}%
\end{align}
Moreover, $\Lambda(t, u_0)$ is compact.
\end{lemma}

\begin{proof}
Let $v(t)=S(t)v_{0}$, $u(t)=S(t)u_{0}$ and $U(t)=\Lambda(t,u_{0})(v_{0}%
-u_{0})$ where $\Lambda(t,u_{0})$ is the solution operator associated to the
linearized equation
\begin{align}
\dot{U}-\Delta U-\nabla\cdot\left( \mu^{\prime}(u)\nabla w(u)U+\mu
(u)\nabla\tilde{w}(U)\right) -g^{\prime}(u)U & =0\text{ in }\Omega
\times(0,\infty)\text{,}\label{linearized eq}\\
(\nabla U+\mu^{\prime}(u)\nabla w(u)U+\mu(u)\nabla\tilde{w}(U))\cdot n &
=0\text{ on }\partial\Omega\times(0,\infty)\text{,}\label{linearized eq 2}%
\end{align}
where $\tilde{w}(U)=K\ast(-2U)$. We can prove that for all $u_{0},v_{0}%
\in\mathcal{A}$, there exists a unique solution $U$ to the linearized equation
(\ref{linearized eq})-(\ref{linearized eq 2}) such that $U(0)=v_{0}-u_{0}$
with the following regularity%
\[
U\in H^{1}(0,T;\left( H^{1}(\Omega)\right) ^{\ast})\cap L^{2}(0,T;H^{1}%
(\Omega))\text{ for all }T>0.
\]
Moreover, for all $t>0$ there exists a constant $c:=c(t)$ independent on
$u_{0}$ and $v_{0}$ such that
\begin{equation}
\left\Vert  \nabla U\right\Vert  _{L^{2}(\Omega)}\leq
c. \label{compactness estimate}
\end{equation}
Aiming at clarity, we postpone the proof of these results to the Appendix (Lemma~\ref{lineq}).

Taking the difference between equations \eqref{nonloc CH 1} for $v$,
\eqref{nonloc CH 1} for $u$ and \eqref{linearized eq} and testing with $v-u-U$,
we get
\begin{align*}
0 & = \frac{\mathrm{d}}{\mathrm{d}t} \frac12 \left\Vert  v-u-U\right\Vert 
_{L^{2}(\Omega)}^{2}+\left\Vert  \nabla(v-u-U)\right\Vert  _{L^{2}(\Omega
)}^{2}\\
& +\int_{\Omega}\left( \mu(v)\nabla w(v)-\mu(u)\nabla w(u)-\mu^{\prime
}(u)\nabla w(u)U-\mu(u)\nabla\tilde{w}(U)\right) \nabla(v-u-U)\\
& +\int_{\Omega}\left( g(v)-g(u)-g^{\prime}(u)U\right) (v-u-U).
\end{align*}
Note that
\begin{align*}
\int_{\Omega}\left( g(v)-g(u)-g^{\prime}(u)U\right) (v-u-U) & \leq\int_{\Omega}g^{\prime}(u)(v-u-U)^{2}+\int_{\Omega}
\frac12 \sup_{s\in\lbrack0,1]}|g^{\prime\prime}(s)| |v-u|^{2}|v-u-U|\\
& \leq C\left\Vert  v-u-U\right\Vert  _{L^{2}(\Omega)}^{2}\\
& + C \left(
\int_{\Omega}|v-u|^{4}+\int_{\Omega}|v-u-U|^{2}\right) \\
& \leq C\left\Vert  v-u-U\right\Vert  _{L^{2}(\Omega)}^{2}+ C\left\Vert 
v-u\right\Vert  _{L^{4}(\Omega)}^{4}\text{.}%
\end{align*}%
By adding and subtracting first $\mu(u) \nabla w(v)$ and then $\mu^{\prime} (u) (v-u) \nabla w(u)$ we also have
\begin{align*}
& \int_{\Omega}\left( \mu(v)\nabla w(v)-\mu(u)\nabla w(u)-\mu^{\prime
}(u)\nabla w(u)U-\mu(u)\nabla\tilde{w}(U)\right) \nabla(v-u-U)\\
& =\int_{\Omega}\left( \left( \mu(v)-\mu(u)\right) \nabla w(v)-\mu
(u)\left( \nabla w(u)-\nabla w(v)\right) -\mu^{\prime}(u)\nabla
w(u)U-\mu(u)\nabla\tilde{w}(U)\right) \nabla(v-u-U)\\
& =\int_{\Omega}(\mu^{\prime}(u)(v-u-U)\nabla w(u)\nabla(v-u-U)+\int_{\Omega
}(\mu^{\prime}(u)(v-u)\nabla\tilde{w}(v-u)\nabla(v-u-U)\\
& +\int_{\Omega}\mu(u)\nabla\tilde{w}(v-u-U)\nabla(v-u-U)+\int_{\Omega}%
\frac12 \sup_{s\in\lbrack0,1]}|\mu^{\prime\prime}(s)| (u-v)^{2}\nabla w(v)\nabla(v-u-U).
\end{align*}
We estimate
\begin{align*}
& \int_{\Omega}(\mu^{\prime}(u)(v-u-U)\nabla w(u)\nabla(v-u-U)\\
& \leq\sup_{s\in\lbrack0,1]}|\mu^{\prime}(s)|~r_{\infty}\left\Vert  u\right\Vert 
_{L^{\infty}\left( \Omega\right) }\left\Vert  v-u-U\right\Vert 
_{L^{2}\left( \Omega\right) }\left\Vert  \nabla(v-u-U)\right\Vert 
_{L^{2}\left( \Omega\right) }\\
& \leq\frac{1}{2}\left\Vert  \nabla(v-u-U)\right\Vert  _{L^{2}\left(
\Omega\right) }^{2}+C\left\Vert  v-u-U\right\Vert  _{L^{2}\left(
\Omega\right) }^{2}\text{,}%
\end{align*}%
Using assumption~\eqref{K3}
\begin{align*}
& \int_{\Omega}(\mu^{\prime}(u)(v-u)\nabla\tilde{w}(v-u)\nabla(v-u-U)\\
& \leq C\sup_{s\in\lbrack0,1]}|\mu^{\prime}(s)|^{2}\int_{\Omega}%
|v-u|^{2}|\nabla\tilde{w}(v-u)|^{2}+\frac{1}{4}\left\Vert  \nabla
(v-u-U)\right\Vert  _{L^{2}\left( \Omega\right) }^{2}\\
& \leq C\left( \int_{\Omega}|v-u|^{4}\right) ^{1/2}\left( \int_{\Omega
}|\nabla\tilde{w}(v-u)|^{4}\right) ^{1/2}+\frac{1}{4}\left\Vert 
\nabla(v-u-U)\right\Vert  _{L^{2}\left( \Omega\right) }^{2}\\
& \leq Cr_{4}^{2}\int_{\Omega}|v-u|^{4}+\frac{1}{4}\left\Vert  \nabla
(v-u-U)\right\Vert  _{L^{2}\left( \Omega\right) }^{2}\text{,}%
\end{align*}%
\begin{align*}
& \int_{\Omega}\mu(u)\nabla\tilde{w}(v-u-U)\nabla(v-u-U)\\
& \leq\sup_{s\in\lbrack0,1]}\mu(s)r_{2}\left\Vert  v-u-U\right\Vert 
_{L^{2}(\Omega)}\left\Vert  \nabla(v-u-U)\right\Vert  _{L^{2}(\Omega)}\\
& \leq C\left\Vert  v-u-U\right\Vert  _{L^{2}(\Omega)}^{2}+\frac{1}%
{8}\left\Vert  \nabla(v-u-U)\right\Vert  _{L^{2}(\Omega)}^{2}\text{,}%
\end{align*}%
\begin{align*}
& \int_{\Omega}\frac{\mu^{\prime\prime}(c)}{2}(u-v)^{2}\nabla w(v)\nabla
(v-u-U)\\
& \leq\left\Vert  \nabla w(v)\right\Vert  _{L^{\infty}(\Omega)}\int_{\Omega
}|u-v|^{2}|\nabla(v-u-U)|\\
& \leq C\int_{\Omega}|u-v|^{4}+\frac{1}{16}\left\Vert  \nabla
(v-u-U)\right\Vert  _{L^{2}\left( \Omega\right) }^{2}\text{.}%
\end{align*}
Combining the above estimates we get
\begin{align}
& \frac{\mathrm{d}}{\mathrm{d}t}\frac{\left\Vert  v-u-U\right\Vert 
_{L^{2}(\Omega)}^{2}}{2}+\frac{1}{16}\left\Vert  \nabla(v-u-U)\right\Vert 
_{L^{2}(\Omega)}^{2} \nonumber \\
& \leq C\left\Vert  v-u\right\Vert  _{L^{4}\left( \Omega\right) }%
^{4}+C\left\Vert  v-u-U\right\Vert  _{L^{2}(\Omega)}^{2}\text{.} \label{formula}
\end{align}
We now recall that $u_{0}$ and $v_{0}$ belong to the global attractor
$\mathcal{A}$. Fix $t>t_0>0$. Since $\mathcal{A}$ is invariant, $u,v\in\mathcal{A}$
and for every $t_0 \in (0, t+1)$, $u$ is bounded in $L^\infty(t_0, t+1, H^2(\Omega))$
(cf. $(iii)$, Thm.~2).
Using the Gagliardo-Nirenberg
interpolation inequality~\cite{Ni}, we have that
\[
\left\Vert  u-v\right\Vert  _{L^{4}(\Omega)}\leq C\left\Vert  u-v\right\Vert 
_{H^{2}(\Omega)}^{a}\left\Vert  u-v\right\Vert  _{L^{2}(\Omega)}^{1-a}\text{,}%
\]
where $a=d/8$.
Therefore, as a consequence of the continuous dependence from initial data, we
have
\[
\left\Vert  u(t)-v(t)\right\Vert  _{L^{4}(\Omega)}\leq C(t_0,t+1) \left\Vert 
u(t)-v(t)\right\Vert  _{L^{2}(\Omega)}^{\frac{8-d}{8}}\leq C(t_0,t+1) e^{Ct}\left\Vert 
u_{0}-v_{0}\right\Vert  _{L^{2}(\Omega)}^{\frac{8-d}{8}}\text{.}%
\]
Substituting into inequality (\ref{formula}) we get
\begin{equation}
\frac{\mathrm{d}}{\mathrm{d}t} \left\Vert  v-u-U\right\Vert  _{L^{2}%
(\Omega)}^{2} \leq C(t_0,t+1)\left\Vert  u_{0}-v_{0}\right\Vert  _{L^{2}%
(\Omega)}^{4\left( \frac{8-d}{8}\right) }+C\left\Vert  v-u-U\right\Vert 
_{L^{2}(\Omega)}^{2}\nonumber
\end{equation}
and hence, applying the Gronwall lemma and
recalling that \mbox{$v_{0}-u_{0}-U(0)=0$}
\[
\left\Vert  v-u-U\right\Vert  _{L^{2}(\Omega)}^{2}\leq C(t_0,t+1)
\left\Vert  u_{0}-v_{0}\right\Vert  _{L^{2}(\Omega)}^{\frac{8-d}{2}%
}\text{.}%
\]
As a consequence, for every $t>0$ fixed we have%
\begin{align}
& \sup_{v_{0}:0<\left\Vert  u_{0}-v_{0}\right\Vert  _{L^{2}(\Omega)}%
\leq\varepsilon}\frac{\left\Vert  S(t)v_{0}-S(t)u_{0}-\Lambda(t,u_{0}%
)(v_{0}-u_{0})\right\Vert  _{L^{2}(\Omega)}}{\left\Vert  u_{0}-v_{0}\right\Vert 
_{L^{2}(\Omega)}}\nonumber\\
& =\sup_{v_{0}:0<\left\Vert  u_{0}-v_{0}\right\Vert  _{L^{2}(\Omega)}%
\leq\varepsilon}\frac{\left\Vert  v(t)-u(t)-U(t)\right\Vert  _{L^{2}(\Omega)}%
}{\left\Vert  u_{0}-v_{0}\right\Vert  _{L^{2}(\Omega)}}\leq C(t) \left\Vert  u_{0}-v_{0}\right\Vert 
_{L^{2}(\Omega)}^{\frac{8-d}{4}-1}\leq C(t) \varepsilon^{\frac{4-d}{4}}.\label{formula 2}%
\end{align}
As long as $d<4$, the right hand side of (\ref{formula 2}) goes to zero as $\varepsilon
\rightarrow0$. This proves~\eqref{unif diff 1}.
The last step of our proof, i.e., the boundedness and compactness of $\Lambda(u_{0},t)$
for any fixed $t$, follows from estimate~\eqref{compactness estimate}.
\end{proof}

As already mentioned, in order to prove that $d_{\mathrm{frac}}(\mathcal{A})$ is finite we are interested in 
keeping track of the evolution of the volume of an infinitesimal cube. More precisely, we focus our attention on the
following quantity:
\begin{align}
\mathcal{TR}_n (\mathcal{A}) = \sup_{u_0 \in \mathcal{A}}
 \sup_{P^{(n)}(0)} \left\langle \mathrm{Tr}(L(t; u_0)P^{(n)}(t)) \right\rangle,
\end{align}
where $\left\langle \cdot \right\rangle$ denotes $\left\langle h \right\rangle = \underset{t\to\infty}{\operatorname{limsup}}\frac{1}{t}\int_{0}^{t}h(s)
\mathrm{d}s$.
Here $L(t,u_{0}):L^{2}(\Omega)\rightarrow\left( L^{2}(\Omega)\right)
^{\ast}=L^{2}(\Omega)$ (whose domain is $H^{1}(\Omega)$) is the linearized evolution operator at time $t$
associated with the initial condition $u_{0}\in\mathcal{A}$ given by
\[
(L(t,u_{0})\phi,\psi)_{L^{2}(\Omega)}=-\int_{\Omega}\nabla\phi\nabla
\psi+\left( \mu^{\prime}(u)\nabla w(u)\phi+\mu(u)\nabla\tilde{w}%
(\phi)\right) \nabla\psi-g^{\prime}(u)\phi\psi\text{,}%
\]
where $u(t)=S(t)u_{0}$, and $P^{(n)}:L^{2}(\Omega)\rightarrow L^{2}(\Omega)$ is
a rank $n$ projection operator onto the subspace of $L^{2}(\Omega)$. 

We now take advantage of the following result proved in~\cite{Ro}:
\begin{theorem}[\hspace{-.03cm}{\cite[Thm.13.16]{Ro}}]
\label{thfda}
 Suppose that $S(t)$ is uniformly differentiable on $\mathcal{A}$ and that
 there exists a $t_0$ such that $\Lambda(t,u_0)$ is compact for all $t \geq t_0$.
 If $\mathcal{TR}_n(\mathcal{A})<0$ then $d_{\mathrm{frac}}(\mathcal{A}) \leq n$.
\end{theorem}
We prove that in our case $\mathcal{TR}_n(\mathcal{A})<0$ for $n$ big enough.
To this aim, we start by fixing orthonormal vectors $\phi_j$, $j=1, \dots, n$, and defining
$P^{(n)}:L^{2}(\Omega)\rightarrow \underset{j=1, \dots, n}{\text{span}}\left( \{ \phi_j \} \right)$
as the standard orthonormal projection.
We then compute for every $\delta>0$ and some $C_\delta >0$
\begin{align*}
\left\langle \mathrm{Tr}(L(t,u_{0})P^{(n)})\right\rangle & = -\left\langle
\mathrm{Tr}\left( -\Delta P^{(n)}\right) \right\rangle +\left\langle
\sum_{j=1}^{n}\int_{\Omega}\phi_{j}^{2}g^{\prime}(u)\right\rangle \\
&\hspace{.47cm}-\left\langle \sum_{j=1}^{n}\int_{\Omega}\left( \mu^{\prime}(u)\nabla
w(u)\phi_{j}+\mu(u)\nabla\tilde{w}(\phi_{j})\right) \nabla\phi_{j}%
\right\rangle \\
& \leq-\left\langle \int_{\Omega}|\nabla\phi_{j}|^{2}\right\rangle +\left\Vert 
g^{\prime}(u)\right\Vert  _{L^{\infty}(0,\infty;L^{\infty}(\Omega
))}\left\langle \sum_{j=1}^{n}\int_{\Omega}\phi_{j}^{2}\right\rangle \\
&\hspace{.47cm}+\left\Vert  \mu^{\prime}(u)\nabla w(u)\right\Vert  _{L^{\infty}%
(0,\infty;L^{\infty}(\Omega))}\left\langle \sum_{j=1}^{n}\left( C_{\delta
}\int_{\Omega}\phi_{j}^{2}+\delta\int_{\Omega}|\nabla\phi_{j}|^{2}\right)
\right\rangle \\
&\hspace{.47cm}+\left\Vert  \mu(u) \right\Vert  _{L^{\infty}(0,\infty;L^{\infty
}(\Omega))}\left\langle \sum_{j=1}^{n}\left( C_{\delta}\int_{\Omega
}\left\Vert  \nabla\tilde{w}(\phi_{j})\right\Vert  ^{2}+\delta\int_{\Omega
}|\nabla\phi_{j}|^{2}\right) \right\rangle \\
& \leq-\left\langle \int_{\Omega}|\nabla\phi_{j}|^{2}\right\rangle +C_{\delta
}\left\langle \sum_{j=1}^{n}\int_{\Omega}\phi_{j}^2\right\rangle +C\delta
\left\langle \sum_{j=1}^{n}\int_{\Omega}|\nabla\phi_{j}|^{2}\right\rangle
\text{.}%
\end{align*}
Note that the constant $C_{\delta}$ depends on $u$ and $u_{0}$ only through
their $L^{\infty}$-norm. Recalling that for all $u_{0}\in\mathcal{A}$, $0\leq
u,u_{0}\leq1$ one can assume without loss of generality (by possibly choosing
a larger constant $C$) that $C_{\delta}$ is independent on $u$ and $u_{0}$.
Choosing $\delta$ sufficiently small, and recalling that $\int_{\Omega}%
\phi_{j}{}^{2}=1$ we have
\begin{align*}
\left\langle \mathrm{Tr}(L(t,u_{0})P^{(n)})\right\rangle & \leq-\frac{1}%
{2}\left\langle \int_{\Omega}|\nabla\phi_{j}|^{2}\right\rangle +Cn\\
& =-\frac{1}{2}\left\langle \mathrm{Tr}\left( -\Delta P^{(n)}\right)
\right\rangle +Cn\text{.}%
\end{align*}
Applying \cite[Lemma 13.17]{Ro}, we get
\[
\left\langle \mathrm{Tr}(L(t,u_{0})P^{(n)})\right\rangle \leq-\frac{1}%
{2}\left\langle \mathrm{Tr}\left( -\Delta P^{(n)}\right) \right\rangle
+Cn\leq-\frac{1}{2}Cn^{\frac{d+2}{d}}+Cn\text{.}%
\]

Thus, there exists $N\in\mathbb{N}$ such that $\left\langle
\mathrm{Tr}(L(t,u_{0})P^{(n)})\right\rangle $ is negative for all $n\geq N$.
Moreover $N$ is independent of the choice of the vectors $\phi_{j}$ and of the initial
condition $u_{0}\in\mathcal{A}$. By virtue of Thm.~\ref{thfda} and
Lemma \ref{lemma unif dif} (see Appendix), we conclude that the fractal dimension of
$\mathcal{A}$ is smaller than $N$, hence finite.


\section{Exponential attractor: Proof of Theorem~\ref{thm:ea}} \label{sec:ea}

We now prove the existence of a finite-dimensional exponential attractor $\mathcal{A}^{\prime }$. The idea we follow
is presented in~\cite{EMZ} to show an analogous result in the case $g=0$. To this aim, we first need to 
prove the following:

\begin{lemma}
Let $u_{01}, u_{02} \in X$ and let $u_1, u_2$ be the solutions of~\eqref{nonlocCH} corresponding to the initial 
conditions $u_1(0)=u_{01}, u_2(0)=u_{02}$. Then, there exists $C$ independent of $t$ such that
\begin{equation} \label{eq:ast}
 \Vert u_1(t)-u_2(t) \Vert_{L^2(\Omega)} \leq C \mathrm{e}^{-t} \Vert u_{01}-u_{02} \Vert_{L^2(\Omega)} + C \Vert
 u_1-u_2 \Vert_{L^2(0,t;L^2(\Omega))},
\end{equation}
\begin{equation} \label{eq:ast2}
 \Vert u_1-u_2 \Vert_{H^1(0,t; H_1^\ast (\Omega)) \cap L^2(0,t;H^1(\Omega))} \leq C \Vert u_{01}-u_{02}
 \Vert_{L^2(\Omega)}.
\end{equation}
\end{lemma}
\begin{proof}
Equation~\eqref{eq:ast} comes from the following. 
Taking the difference of the equation~\eqref{nonloc CH 1}
for $u_{1}$ and $u_{2}$ and testing it with $u_{1}-u_{2}$ one can easily
obtain (see \cite[Sec.~3.1]{MR} for details)
\[
\frac{\mathrm{d}}{\mathrm{d}t}\left\Vert u_{1}(t)-u_{2}(t)\right\Vert
_{L^{2}(\Omega )}^{2}+\frac{1}{2}\left\Vert \nabla u_{1}(t)-\nabla
u_{2}(t)\right\Vert _{L^{2}(\Omega )}^{2}\leq C\left\Vert
u_{1}(t)-u_{2}(t)\right\Vert _{L^{2}(\Omega )}^{2}.
\]%
In particular,
\[
\frac{\mathrm{d}}{\mathrm{d}t}\left\Vert u_{1}(t)-u_{2}(t)\right\Vert
_{L^{2}(\Omega )}^{2}+\left\Vert u_{1}(t)-u_{2}(t)\right\Vert _{L^{2}(\Omega
)}^{2}\leq \left( C+1\right) \left\Vert u_{1}(t)-u_{2}(t)\right\Vert
_{L^{2}(\Omega )}^{2}
\]%
which yields, thanks to the Gronwall Lemma, 
\begin{equation}
\left\Vert u_{1}(t)-u_{2}(t)\right\Vert _{L^{2}(\Omega )}^{2}\leq \mathrm{e}%
^{-t}\left\Vert u_{1,0}-u_{2,0}\right\Vert _{L^{2}(\Omega
)}^{2}+(C+1)\int_{0}^{t}\left\Vert u_{1}-u_{2}\right\Vert _{L^{2}(\Omega
)}^{2}, \label{ee1}
\end{equation}%
that is equivalent to~\eqref{eq:ast}.
In order to derive~\eqref{eq:ast2}, we first recall that \cite[Sec.~3.1]{MR}
\begin{equation} \label{ee3}
\left\Vert u_{1}(t)-u_{2}(t)\right\Vert _{L^{2}(\Omega
)}^{2}+\int_{0}^{t}\left\Vert \nabla u_{1}-\nabla u_{2}\right\Vert
_{L^{2}(\Omega )}^{2}\leq C\mathrm{e}^{Ct}\left\Vert
u_{0,1}-u_{0,2}\right\Vert _{L^{2}(\Omega )}^{2}.
\end{equation}
Moreover, by testing the
difference of equation~\eqref{nonloc CH 1} for $u_1$ and $u_2$ with a test function $\psi \in H^{1}(0,T;H^{1}(\Omega
))$, and using the Lipschitz property of $g$ and $\mu $ and the linearity of
the convolution, we get:
\begin{eqnarray*}
&&\left\langle \partial _{t}u_{1}-\partial _{t}u_{2},\psi \right\rangle
_{H^{1}(\Omega )} \\
&=&\left( \nabla (u_{1}-u_{2})+\mu (u_{1})\nabla w(u_{1})-\mu (u_{2})\nabla
w(u_{2}),\nabla \psi \right) _{L^{2}(\Omega )}+\left( g(u_{1})-g(u_{2}),\psi
\right) _{L^{2}(\Omega )} \\
&\leq &C\left\Vert u_{1}-u_{2}\right\Vert _{H^{1}(\Omega )}\left\Vert \psi
\right\Vert _{H^{1}(\Omega )}.
\end{eqnarray*}%
Then, using this result together with equation~\eqref{ee3} we finally obtain
\begin{equation} \label{ee2}
\left\Vert \partial _{t}u_{1}-\partial _{t}u_{2}\right\Vert _{L^{2}\left(
0,t;\left( H^{1}(\Omega )\right) ^{\ast }\right)
}^{2}+\int_{0}^{t}\left\Vert u_{1}-u_{2}\right\Vert _{H^{1}(\Omega
)}^{2}\leq C\mathrm{e}^{Ct}\left\Vert u_{0,1}-u_{0,2}\right\Vert
_{L^{2}(\Omega )}^{2}.
\end{equation}%
\hspace{15cm}
\end{proof}

Now we can proceed with proving Theorem~\ref{thm:ea}.
We choose $T>0$ such that $\gamma := C \mathrm{e}^{-T} < \frac12$ and define
\begin{equation}
\begin{aligned}
 \mathcal{H}_1 &:= H^1(0,t; H_1^\ast(\Omega)) \cap L^2(0,t;H^1(\Omega)), \\
 \mathcal{H} &:= L^2(0,t;L^2(\Omega)), \\
 \mathbb{T} &: u_0 \in X \mapsto u \in \mathcal{H},
\end{aligned}
\end{equation}
where $u$ solves~\eqref{nonlocCH} with $u(0)=u_0$. 
Note that equations~\eqref{eq:ast} and~\eqref{eq:ast2} can be rewritten as
\begin{equation} \label{eq:piu}
 \Vert S(T)u_{01} - S(T)u_{02} \Vert_{L^2(\Omega)} \leq \gamma \Vert u_{01}-u_{02} \Vert_{L^2(\Omega)} + \tilde{C}
 \Vert  \mathbb{T}u_{01}-\mathbb{T}u_{02} \Vert_\mathcal{H},
\end{equation}
\begin{equation} \label{eq:piu2}
 \Vert \mathbb{T} u_{01} - \mathbb{T} u_{02} \Vert_{\mathcal{H}_1} \leq C \mathrm{e}^{C T} \Vert u_{01} - u_{02}
 \Vert_{L^2(\Omega)},
\end{equation}
for some $\tilde{C}$ independent of $T$.\\
\noindent
We now build a discrete exponential attractor for the (discrete) system $(X, S(T))$. 
Define $X_0 = \left\{ 0 \right\}$. By definition of $X$ we have that $X \subseteq \mathcal{B}_R^{L^2(\Omega)}(0)$,
where
$R = |\Omega|$ and $\mathcal{B}_R^{L^2(\Omega)}(0)$ denotes the $L^2(\Omega)$-ball of radius $R$ centered in $0$.
As a consequence of~\eqref{eq:piu2} we have
\begin{equation} \label{eq:xx}
 \mathbb{T}(X) \subseteq \mathcal{B}_{R(C \mathrm{e}^{CT})}^{\mathcal{H}_1} (\mathbb{T}(0)).
\end{equation}
Fix now $\theta>0$ such that $2(\gamma + \theta)<1$. Since $\mathcal{H}_1$ is compactly embedded in $\mathcal{H}$, 
there exists a finite number $n_1$ of $\mathcal{H}$-balls of radius $R \frac{\theta}{\tilde{C}}$ that cover $\mathbb{T}(X)$,
i.e., $\exists~X_1 = \left\{ x_1, \dots, x_{n_1} \right\} \subseteq X^{n_1}$ such that
\begin{equation} \label{eq:cross}
 \mathbb{T}(X) \subseteq \bigcup_{x \in X_1} \mathcal{B}_{R \frac{\theta}{\tilde{C}}}^{\mathcal{H}}(\mathbb{T}(x)).
\end{equation}
Without loss of generality we assume that $n_1$ is the minimum number for which~\eqref{eq:cross} holds. Note that,
as a consequence of~\eqref{eq:piu2}, $n_1$ can be estimated by the minimum number of $\mathcal{H}$-balls of radius
$R \frac{\theta}{\tilde{C}}$ necessary to cover $\mathcal{B}_{R(C \mathrm{e}^{CT})}^{\mathcal{H}_1} (0)$. Thus, $n_1
\leq N(\theta, T)$, where $N(\theta, T)$ depends on $\theta, T, \tilde{C}$ but not on $R$. \\
\noindent
We now observe that the family of $L^2(\Omega)$-balls with radius $2(\gamma + \theta)R$ and centers $S(T)x_i \in 
X_1$ covers $S(T)X$. Indeed, let $x \in S(T)X$. Then, there exists $u_0 \in X$ such that $S(T)u_0 = x$. Thanks
to~\eqref{eq:piu}, we estimate
\[
 \Vert S(T) u_0 - S(T) x_i \Vert_{L^2(\Omega)} \leq \gamma \Vert u_0 - x_i \Vert_{L^2(\Omega)} + \tilde{C} \Vert \mathbb{T}
 u_0 - \mathbb{T} x_i \Vert_\mathcal{H} \leq 2 (\gamma + \theta) R.
\]
Here we used that (recalling~\eqref{eq:cross})
\[
 \tilde{C} \mathbb{T} (u_0) \in \bigcup_{x \in X_1} \mathcal{B}_{R \theta}^{\mathcal{H}}(\mathbb{T}(x)) \subseteq
 \bigcup_{x \in  X_1} \mathcal{B}_{(\gamma + \theta) R}^{\mathcal{H}}(\mathbb{T}(x)).
\]
Hence, 
\[
 S(T) X \subseteq \bigcup_{x \in X_1} \mathcal{B}_{2 (\gamma + \theta) R}^{L^2(\Omega)}(x).
\]
Applying the same procedure to each ball in the covering $\left\{ \mathcal{B}_{2 (\gamma + \theta) R}^{L^2(\Omega)}(x) : x \in X_1 \right\}$
we can find a new covering of $S(T)^2X$ with at most $N^2(\theta,T)$ $L^2(\Omega)$-balls of radius $(2(\gamma + \theta))^2R$,
i.e., there exists $X_2$, $\#X_2 \leq N^2(\theta,T)$, such that
\[
 S(2T)X = S(T)^2 X \subseteq \bigcup_{x \in X_2} \mathcal{B}_{(2 (\gamma + \theta))^2 R}^{L^2(\Omega)}(x).
\]
Iterating this argument for all $K \in \mathbb{N}$ we can show that there exists a set $X_K \in X$ such that
$\#X_K \leq N^K(\theta, T)$ and
\[
 S(KT)X = S(T)^K X \subseteq \bigcup_{x \in X_K} \mathcal{B}_{(2 (\gamma + \theta))^K R}^{L^2(\Omega)}(x).
\]
In particular,
\[
 d_{L^2(\Omega)} \left( S(T)^K X,X_K \right) \leq R \left( 2(\gamma + \theta) \right)^K.
\]
The set $\mathcal{A}_d=\bigcup_{K=1}^{+\infty} X_K$ is then an exponential attractor for the discrete system $(S(T),X)$. \\
\noindent
Moreover, $\mathcal{A}_d$ has finite fractal dimension. Indeed, fix $\varepsilon>0$ and let$N_\varepsilon$ be the minimum
number of $L^2(\Omega)$-balls of radius $\varepsilon$ needed to cover $\mathcal{A}_d$.
Let $K \in \mathbb{N}$ be the integer part of $\log_{2(\gamma+\theta)}(\varepsilon) = \frac{\log \varepsilon}{\log (2(\gamma+\theta))}$,
i.e., $(2(\gamma+\theta))^{K+1} < \varepsilon \leq (2(\gamma+\theta))^K$.
Thus, $N_\varepsilon \leq N_{(2(\gamma+\theta))^{K+1}}=(N(\theta,T))^{K+1}$.
Consequently
\[
 \underset{\varepsilon \to 0}{\operatorname{limsup}} \frac{\log N_\varepsilon}{\log \frac{1}{\varepsilon}} \leq \frac{(K+1)\log N(\theta, T)}{-\log \varepsilon}
 =-\frac{\log N(\theta,T)}{\log 2(\gamma+\theta)} < +\infty.
\]
As a consequence of Definition~\ref{deffracdim}, $d_{\mathrm{frac}} \mathcal{A}_d$ is finite.\\
\noindent
Thus, there exists a discrete exponential attractor $%
\mathcal{A}_{d}$, i.e.,
\[
d_{L^{2}(\Omega )}(\left( S(T)\right) ^{n}u_{0},\mathcal{A}%
_{d})=d_{L^{2}(\Omega )}(S(nT)u_{0},\mathcal{A}_{d})\leq C\mathrm{e}^{-nc}%
\text{ for all }u_{0}\in X
\]%
for some positive constant $C$. We now construct the exponential attractor
for the semigroup $S(t)$. To this aim, we define
\[
\mathcal{A}^{\prime }=U_{t\in \lbrack 0,T]}S(t)\mathcal{A}_{d}.
\]%
For all $t>0$ we can find $n$ such that $t=nT+\tilde{t}$ where $0\leq \tilde{%
t}<T$. For all $\varepsilon >0$ we can
find $a_{n,u_{0},\varepsilon }\in \mathcal{A}_{d}$ such that
\[
\left\Vert S(nT)u_{0}-a_{n,u_{0},\varepsilon }\right\Vert _{L^{2}(\Omega
)}\leq \left( C+\varepsilon \right) \mathrm{e}^{-nc}.
\]%
By virtue of estimate (\ref{ee3}), we have that
\[
\left\Vert S(t)u_{0}-S(\tilde{t})a_{n,u_{0},\varepsilon }\right\Vert
_{L^{2}(\Omega )}\leq C\mathrm{e}^{C\tilde{t}}\left\Vert
S(nT)u_{0}-a_{n,u_{0},\varepsilon }\right\Vert _{L^{2}(\Omega )}\leq C%
\mathrm{e}^{CT}\left( C+\varepsilon \right) \mathrm{e}^{-nc}.
\]%
Since $T$ is fixed, by suitably renaming the constants we have%
\[
\left\Vert S(t)u_{0}-S(\tilde{t})a_{n,u_{0},\varepsilon }\right\Vert
_{L^{2}(\Omega )}\leq \left( C+\varepsilon \right) \mathrm{e}^{-ct}.
\]%
Since, by construction $S(\tilde{t})a_{n,u_{0},\varepsilon }\in \mathcal{A}%
^{\prime }$, by the arbitrariness of $\varepsilon $ we conclude that $%
\mathcal{A}^{\prime }$ is an exponential attractor for $S(t)$.

\section{Equilibria: Proof of Theorem~\ref{thm:eq}}

\subsection{Existence}

Since we are looking for solutions $0\leq u\leq1$, we can assume without loss of generality
(and using hypothesis (G))
\begin{align}
 & \mu(s)=0 \hspace{.3cm} \forall s\notin [0,1] \label{ass mu}\\
 & g(s)=g(0) \geq 0 \text{ for } s<0 \text{ and } g(s)=g(1) \leq 0 \text{ for } s>1 \qquad \text{a.e.~in }\Omega. \label{ass g} 
\end{align}
%
The proof is based on a regularization procedure and a fixed point argument.
We first consider the regularization problem parametrized by small $\varepsilon>0$:
\begin{subequations}
\label{eqeqreg}
\begin{align}
-\Delta u-\nabla\cdot(\mu\nabla w)+\varepsilon u  &  =g(u), \label{reg eq}\\
n\cdot(\mu\nabla w+\nabla u)  &  =0\text{,}\\
w  &  =K\ast(1-2u). \label{reg eq 3}%
\end{align}
\end{subequations}
We define then the mapping $\Gamma:z\in L^{2}(\Omega)\mapsto u\in L^{2}\left(  \Omega\right)$
where $u$ is weak solution to
\begin{equation}
\label{fix point eq}%
-\Delta u+\varepsilon u=\nabla\cdot(\mu(z)\nabla(K\ast(1-2z)))+g(z)\text{.}
\end{equation}
The hypotheses on $\mu$, $g$ and $K$ ensure 
that the right hand side of~\eqref{fix point eq} is in $\left(  H^{1}(\Omega)\right)  ^{\ast}$
and $\Gamma$ is well defined thanks to the Lax-Milgram theorem.
In order to prove the existence of a fixed point, we apply Schaefer's fixed point Theorem (Thm.~\ref{Sch},
see Appendix). The continuity of $\Gamma$ is again guaranteed by the 
assumptions on $K$ and the fact that $\mu$ and $g$ are Lipschitz and bounded.
Compactness of $\Gamma$ can be proved by testing~\eqref{fix point eq} with $u$
and using \mbox{$\Vert (\mu(z)\nabla(K\ast
(1-2z)))+g(z)\Vert _{L^{2}(\Omega)}\leq C\Vert z\Vert _{L^{2}(\Omega)}+C$}
to obtain the estimate
\[
\Vert u\Vert _{H^{1}(\Omega)}\leq C_{\varepsilon}\Vert z\Vert _{L^{2}(\Omega)}+C_{\varepsilon}\text{.}%
\]
In order to apply Schaefer's fixed point theorem, we are now only left with proving that the set $\{u\in L^{2}\left(
\Omega\right)  :u=\alpha \Gamma(u)~$for some $\alpha
\in\lbrack0,1]\}$ is bounded. This is equivalent to showing that the set
\[
A:=\{u\in
L^{2}\left(  \Omega\right)  :u/\alpha=\Gamma(u) \text{ for some } \alpha\in(0,1]\}
\]
is bounded. To this end, let $u\in A$. Then, there exists $\alpha\in(0,1]$ such that
\begin{equation}
-\Delta u/\alpha+\varepsilon u/\alpha=\nabla\cdot(\mu(u)\nabla(K\ast(1-2u)))+g(u).
\label{alpha eq}%
\end{equation}
By testing the equation with $\alpha(u-1)^{+}$ and using assumptions~\eqref{ass mu}-\eqref{ass g}, we get%
\[
\int_{u\geq1}|\nabla u|^{2}+\varepsilon\int_{u\geq1}\left(  u^{2}-u\right)
=-\alpha\int_{u\geq1}\mu\nabla w\nabla u+\alpha\int_{u\geq1}g(u)(u-1)\leq0.
\]
This yields $u\leq1$ a.e.~in $\Omega$ for every $u\in A$. Similarly, it can be shown that
$u\geq0$ a.e.~in $\Omega$ for every $u\in A$.
This implies that $A$ is bounded in $L^{2}(\Omega)$ for every $\varepsilon$ fixed.
Thus, the mapping $\Gamma$ has a fixed point $u$ which
solves~\eqref{eqeqreg} and such that $0\leq u\leq1$ a.e. in $\Omega$.

\noindent
Our final step is the passage to the limit $\varepsilon \to 0$.
Let $u_{\varepsilon}$ be a solution of the regularized problem~\eqref{eqeqreg}; hence,
$0\leq u_{\varepsilon}\leq1$. Test~\eqref{reg eq} with $u_{\varepsilon}$, and estimate%
\begin{align*}
\int_{\Omega}|\nabla u_{\varepsilon}|^{2}+\varepsilon|u_{\varepsilon}|^{2}  &
=\int_{\Omega}\mu(u_{\varepsilon})\nabla w_{\varepsilon}\nabla u_{\varepsilon}%
+\int_{\Omega}g(u_{\varepsilon})u_{\varepsilon}\\
&  \leq C+C\Vert \nabla u_{\varepsilon}\Vert _{L^{2}(\Omega)}\text{.}%
\end{align*}
This implies $\Vert \nabla u_{\varepsilon}\Vert _{L^{2}(\Omega)}\leq C$. Thanks to
$0\leq u_{\varepsilon}\leq1$, we get $\Vert u_{\varepsilon}\Vert _{H^{1}(\Omega)}\leq C$, where $C$
is independent of $\varepsilon$. Consequently, for a not relabeled subsequence
we get
\begin{align*}
u_{\varepsilon}  &  \rightarrow u\text{ weakly in }H^{1}\left(  \Omega\right), \\
u_{\varepsilon}  &  \rightarrow u\text{ strongly in }L^{2}\left(  \Omega\right)
\ \text{and pointwise a.e. in }\Omega\text{.}%
\end{align*}
Moreover, using the continuity and boundedness of $\mu$ and $g$ we have
\begin{align*}
\mu(u_{\varepsilon})  &  \rightarrow\mu(u)\text{ strongly in }L^{2}\left(
\Omega\right), \\
g(u_{\varepsilon})  &  \rightarrow g(u)\text{ strongly in }L^{2}\left(
\Omega\right).
\end{align*}
Finally, thanks to the continuity of the convolution and assumption~\eqref{K3},%
\begin{align*}
w_{\varepsilon}=K\ast(1-2u_{\varepsilon}) &\rightarrow w=K\ast(1-2u)\text{ strongly in
}H^{1}(\Omega), \\
\mu(u_{\varepsilon})\nabla w_{\varepsilon} &\rightarrow \mu(u)\nabla w\text{ weakly in
}H^{1}(\Omega)\text{. }%
\end{align*}
This allows us to pass to the limit in the weak formulation of~\eqref{eqeqreg}
and prove that $u$ solves~\eqref{eqeq} as well as $0\leq u\leq1$.

\begin{remark} \label{rem:nue}
Uniqueness of equilibrium points is not guaranteed. As an example,
consider a function $g$ such that $g(0)=g \left( \frac{1}{2} \right) = g(1)=0$. The
constant functions $u^{\ast}(x)=0$, $u^{\ast}(x)= \frac{1}{2}$ and
$u^{\ast}(x)=1$ are then all possible solutions of~\eqref{eqeq}.
\end{remark}

\subsection{Convergence to equilibria}

As stated in the introduction, many known results about convergence to 
equilibria for the reaction-free Cahn-Hilliard equation rely on the existence
of a Lyapunov functional~\cite{AW, LP1, LP2}. The presence of 
a reaction term, however, does not allow us to apply the same techniques.
Therefore, we adopt two different strategies to obtain the convergence to equilibria in
the following cases:
\begin{enumerate}
\item $g$ has sign, i.e. $g\geq0$, or $g\leq0$ (more precisely, case $1$ of Theorem~\ref{thm:eq})

\item $g$ is monotone decreasing and $g^{\prime}\leq-\lambda$ for a
sufficiently large positive constant $\lambda$ (case $2$ of Theorem~\ref{thm:eq}).
\end{enumerate}

\noindent
In the first case, we observe that the total mass is monotone as a function of time. Combined with the
boundedness of $u$, this proves the convergence of the solution to one of the pure phases.
In the second case we exploit a linearization technique around the equilibrium
point. \\
\noindent
We emphasize that these conditions include a wide spectrum of
reaction terms notably relevant in applications, such as~\eqref{bio react term}-\eqref{cho}.  In
particular we prove convergence to equilibria for the
Cahn-Hilliard-Bertozzi~\cite{BEG} and the Cahn-Hilliard-Oono~\cite{BO} equations, as well as for a Cahn-Hilliard
equation coupled with a logistic-type reaction term~\cite{KS}.

\subsubsection{Case 1.}
We outline the details for the convergence to $u=1$. A similar approach can be applied
to show the second part of the theorem (i.e., the convergence to $u=0$).
Hypothesis~\eqref{hA1} and Theorem~\ref{existence and separation} guarantee
that $u$ separates from $0$ uniformly in time, namely
$\forall~T_0\ \exists~k_1(\bar{u}_{0},T_{0}) > 0$ such that $u(t)\geq k_1$ for all
$t\geq T_{0}$ a.e.~\mbox{in $\Omega$}.\footnote{ 
Note that $k_{1}$ is strictly positive for every $u_{0}$ if
$g(x,0)>0$ on a set of positive measure, and for every $u_{0}\neq0$ (a.e.~in
$\Omega$) if $g(x,0)=0$ a.e.~in $\Omega$.}
Using hypothesis~\eqref{hA2} with $k_1=\kappa$
and defining $r_\kappa (s) := -m_\kappa(s-1)$, we have that $g(x,s) \geq r_\kappa(s) \ \forall s \in \left[ \kappa, 1
\right]$. \\
\noindent
Testing equation~\eqref{nonloc CH 1} with $\frac{1}{|\Omega|}$, we get
\[
\frac{\mathrm{d}}{\mathrm{d}t} \bar{u} = \frac{1}{|\Omega|}\int_{\Omega}g(x,u)
\geq \frac{1}{|\Omega|} \int_\Omega r_\kappa(u)=r_\kappa \left( \bar{u} \right).
\]
Recalling $r_\kappa(s)=-m_\kappa(s-1)$, we can rewrite the last inequality as
$\frac{\mathrm{d}}{\mathrm{d}t} \left( \bar{u}-1 \right) \geq -m_\kappa \left( \bar{u}-1 \right).$
Since $0 \leq u \leq 1$ (cf.~Theorem~\ref{existence and separation}), we have then
$0 \geq \bar{u}(t)-1 \geq \left( \bar{u}_0 - 1 \right) \exp (-m_\kappa t)$. Hence,
$\bar{u}(t) \to 1$. Finally, applying Lemma~\ref{l2 convergence} to $U=1-u$ we get
$u(t)\rightarrow1$ in $L^{2}(\Omega)$. 

This proves $u(t)\rightarrow1$ exponentially fast in $L^{2}(\Omega)$ for every $u_{0} \in X$ if the measure of 
the set $\{x \in X \text{ s.t. } g(x,0) > 0\}$ is positive and for
every $u_{0}\neq0$ if $g(\cdot,0)=0$. We recall that, if $g(0)=0$,
$u=0$ is an equilibrium point.

The previous result can be applied to show the convergence to equilibria for the 
Cahn-Hilliard equation with logistic reaction term~\eqref{bio react term} (cf. Corollary~\eqref{cor:log}).

\vspace{1cm}

\begin{figure}[!ht]
\label{fig:gcaso1}
\centering
\includegraphics[scale=1.2]{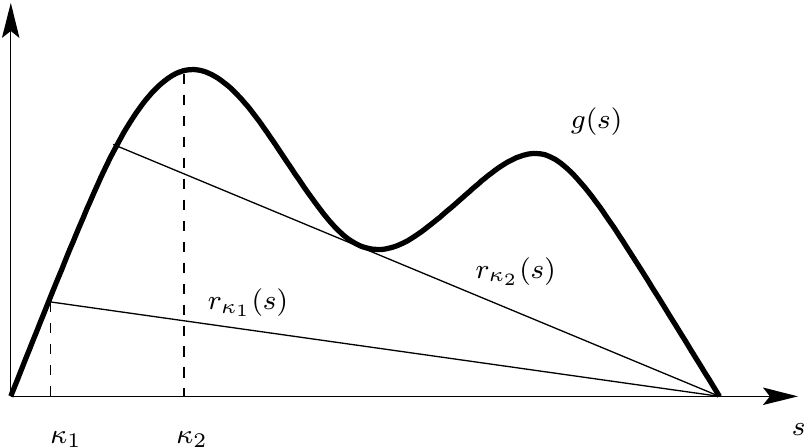}
\caption{Illustration of $r_\kappa$ for a reaction term $g$ satisfying~\eqref{hA1}, \eqref{hA2}.}
\end{figure}

\subsubsection{Case 2.}
Let $u$ be a solution to~\eqref{nonloc CH 1} and $u^{\ast}$ be an equilibrium point.
Taking the difference between equations~\eqref{nonloc CH 1} and \eqref{equilibrium equation}
and testing it with $U=u-u^{\ast}$ gives
\begin{equation}
\frac{\mathrm{d}}{\mathrm{d}t}\frac{1}{2}\Vert U\Vert _{L^{2}(\Omega)}^{2}+\Vert \nabla
U\Vert _{L^{2}(\Omega)}^{2}+\int_{\Omega}\left(  \mu\nabla w-\mu^{\ast}\nabla
w^{\ast}\right)  \nabla U=\int_{\Omega}(g(u)-g(u^{\ast}))U\text{,} \label{qq}%
\end{equation}
where $\mu^{\ast}=\mu(u^{\ast})$ and $w^{\ast}=K\ast(1-2u^{\ast})$. Thanks to~\eqref{K3},
we estimate
\begin{align*}
\int_{\Omega}\left(  \mu\nabla w-\mu^{\ast}\nabla w^{\ast}\right)  \nabla U
&  =\int_{\Omega}\left(  \mu\nabla\left(  w-w^{\ast}\right)  \right)  \nabla
U+\int_{\Omega}\left(  \mu-\mu^{\ast}\right)  \nabla w^{\ast}\nabla U\\
&  \leq\frac{1}{4}r_{2}\Vert U\Vert _{L^{2}(\Omega)}\Vert \nabla U\Vert _{L^{2}(\Omega)}%
+\int_{\Omega}U(1-u-u^{\ast})\nabla w^{\ast}\nabla U\\
&  \leq \left( \frac{r_2}{4} + r_\infty \right) \Vert U \Vert_{L^2(\Omega)} \Vert \nabla U \Vert_{L^2(\Omega)}\\
& \leq \frac12 \Vert \nabla U \Vert_{L^2(\Omega)}^2 + \frac12 \left( \frac{r_2}{4} + r_\infty \right)^2
\Vert U \Vert_{L^2(\Omega)}^2.%
\end{align*}
Here we used the fact that $\underset{s \in  [0,1]}{\text{sup}} \mu(s) \leq \frac14$.
By assumption~\eqref{hC} we have
\[
\int_{\Omega}(g(u)-g(u^{\ast}))U\leq-\lambda\Vert U\Vert _{L^{2}(\Omega)}^{2}\text{.}%
\]
Thus, by substituting into~\eqref{qq}, we get
\[
\frac{\mathrm{d}}{\mathrm{d}t}\frac{1}{2}\Vert U\Vert _{L^{2}(\Omega)}^{2}+\frac{1}%
{2}\Vert \nabla U\Vert _{L^{2}(\Omega)}^{2}\leq\left(  C_{1}-\lambda\right)
\Vert U\Vert _{L^{2}(\Omega)}^{2},%
\]
where $c_1 = \frac12 \left( \frac{r_2}{4} + r_\infty \right)^2$.
By applying the Gronwall Lemma one gets%
\[
\Vert U(t)\Vert _{L^{2}(\Omega)}^{2}\leq e^{-2\left(  C_{1}-\lambda\right)  t}%
\Vert u_{0}-u^{\ast}\Vert _{L^{2}(\Omega)}^{2}\text{.}%
\]
Thus, for $\lambda>C_{1}$ we have $U(t)\rightarrow0$ in $L^{2}(\Omega)$. This implies $u(t) \to u^*$ in
$L^{2}(\Omega)$.


\section*{Appendix}

In the following we include some auxiliary results we use throughout our proofs.

\begin{theorem}[{Schaefer \cite[Thm.4, Sec.9.2]{Ev}}]
\label{Sch}
Let $X$ be Banach and
let \mbox{$\Gamma:X\rightarrow X$} be continuous, compact, and such that the set $\{x\in
X:x=\alpha \Gamma(x)~$for some $\alpha\in\lbrack0,1]\}$ is bounded. Then, $\Gamma$ has a
fixed point.
\end{theorem}

\begin{lemma}[{Uniform Gronwall Lemma, \cite[Ex.11.2]{Ro}}]
\label{ugl}
Let $\eta$ be an absolutely continuous nonnegative
function on $[t_{0},+\infty)$ and $\phi$,$\psi\in L_{loc}^{1}(t_{0},+\infty)$
two a.e. nonnegative functions such that
\[
\dot{\eta}(t)\leq\phi(t)\eta(t)+\psi(t)\text{ a.e. in }[t_{0},+\infty)
\]
with
\begin{align*}
\int_{t}^{t+r}\phi(s)ds  &  \leq a_{1},~\int_{t}^{t+r}\psi(s)ds\leq
a_{2},~\int_{t}^{t+r}\eta(s)ds\leq a_{3}
\end{align*}
for some positive constants $r$ and $a_{j}$, $j=1,2,3$, and $\forall t  \geq t_{0}$. Then,
\[
\eta(t)\leq\left(  \frac{a_{3}}{r}+a_{2}\right)  e^{a_{1}} \hspace{.7cm} \forall
t\geq t_{0}+r\text{.}%
\]
\end{lemma}

\begin{lemma}[$L^2$-convergence] \label{l2 convergence}
Let $u\in C\left(  [0,T],L^{2}(\Omega)\right)  $, $0\leq
u\leq1$ a.e.~in $\Omega\times\lbrack0,+\infty)$, and $\bar{u}(t)=\int_{\Omega
}u(t)\rightarrow0$ for $t\rightarrow+\infty$. Then, $u(\cdot,t)\rightarrow0$
in $L^{2}(\Omega)$. Moreover, if $\bar{u}(t) \to 0$ exponentially fast,
then also $\Vert u(t) \Vert_{L^2(\Omega)} \to 0$ exponentially fast.
\end{lemma}

\begin{proof}
We start by observing that 
$$
\int_\Omega (u(t)-\bar{u}(t))^+=\int_\Omega (u(t)-\bar{u}(t))^- \leq |\Omega| \bar{u}(t) \to 0.
$$
This implies that $\Vert u(t)-\bar{u}(t) \Vert_{L^1(\Omega)}\leq 2|\Omega| \bar{u}(t) \to 0$. Since $|u(t)-\bar{u}(t)| \leq 1$ a.e. in $\Omega$, we have that 
$$
\Vert u\left(  t\right)  -\bar{u}\left(  t\right)  \Vert _{L^{2}(\Omega)}^{2}  = \int_\Omega |u(t)-\bar{u}(t)|^2 \leq
\int_\Omega |u(t)-\bar{u}(t)| \leq 2|\Omega| \bar{u}(t) \to 0.
$$
This concludes the proof of the lemma.
\end{proof}


\begin{lemma}[Linearized equation] \label{lineq}
Let $u_{0},v_{0}\in\mathcal{A}$. There exists a
unique solution $U$ to the equation
\begin{align}
\dot{U}-\Delta U-\nabla\cdot\left( \mu^{\prime}(u)\nabla w(u)U+\mu
(u)\nabla\tilde{w}(U)\right) -g^{\prime}(u)U & =0\text{ in }\Omega
\times(0,\infty)\text{,} \label{E1} \\
(\nabla U+\mu^{\prime}(u)\nabla w(u)U+\mu(u)\nabla\tilde{w}(U))\cdot n &
=0\text{ on }\partial\Omega\times(0,\infty)\text{,}\\
U(0) & =v_{0}-u_{0}\text{,} \label{E2}%
\end{align}
where $\tilde{w}(U)=K\ast(-2U)$ with the following regularity:%
\begin{equation}
U\in H^{1}(0,T;\left( H^{1}(\Omega)\right) ^{\ast})\cap L^{2}(0,T;H^{1}%
(\Omega))\text{ for all }T>0.\label{regU}%
\end{equation}
Moreover, for all $t>0$ there exists a constant $c:=c(t)$ independent of
$u_{0}$ and $v_{0}$ such that
\begin{equation}
\left\Vert  \nabla U\right\Vert  _{L^{2}(\Omega)}\leq c\label{bb}.%
\end{equation}
\end{lemma}

\begin{proof}
The existence of solutions can be obtained by using a Galerkin approximation
scheme and by virtue of the following estimate obtained by testing the
equation with $U$:%
\begin{align}
\frac{1}{2}\frac{\mathrm{d}}{\mathrm{d}t}\left\Vert U\right\Vert
_{L^{2}(\Omega)}^{2}+\left\Vert \nabla U\right\Vert _{L^{2}(\Omega)}^{2} &
\leq\int_{\Omega}\left( \mu^{\prime}(u)\nabla w(u)U+\mu(u)\nabla\tilde
{w}(U)\right) \nabla U+g^{\prime}(u)U^{2}\nonumber\\
& \leq C\int_{\Omega}|U||\nabla U|+|\nabla\tilde{w}(U)||\nabla U|+|U|^{2} \nonumber\\
& \leq C\left\Vert U\right\Vert _{L^{2}(\Omega)}^{2}+\frac{1}{2}\left\Vert \nabla
U\right\Vert _{L^{2}(\Omega)}^{2}\text{.} \label{aa}%
\end{align}
Here we used boundedness of $\mu$, $\mu^{\prime}$, $g^{\prime}$, and
assumption (K3). From estimate (\ref{aa}) it follows, thanks to the
Gronwall lemma,
\begin{align}
\left\Vert U\right\Vert _{L^{2}(\Omega)}^{2} & \leq\left\Vert
U(0)\right\Vert _{L^{2}(\Omega)}^{2}C\mathrm{e}^{Ct}\text{,}\label{est}\\
\int_{0}^{t}\left\Vert \nabla U\right\Vert _{L^{2}(\Omega)}^{2} & \leq
Ct\mathrm{e}^{Ct}\text{,}%
\end{align}
and, by comparison in the equation~\eqref{E1}, $\dot{U}\in L^{2}(0,T;\left( H^{1}(\Omega)\right) ^{\ast}).$\\
\noindent
Uniqueness is a consequence of estimate (\ref{est}) and of the linearity of
the equation. We now prove (\ref{bb}). Fix $t>0$ and integrate (\ref{aa})
between $0$ and $t$, getting%
\begin{equation}
\frac{1}{2}\int_{0}^{t}\left\Vert \nabla U\right\Vert _{L^{2}(\Omega)}^{2}\leq
C\int_{0}^{t}\left\Vert U\right\Vert _{L^{2}(\Omega)}^{2}+\left\Vert
U(0)\right\Vert _{L^{2}(\Omega)}^{2}\text{.} \label{dd}%
\end{equation}
We now test equation~\eqref{E1} with $\mathbf{(-}\Delta U\mathbf{)}$.
Using boundedness of $u$, $\mu$, $\mu^{\prime}$, $\mu^{\prime\prime}$,
assumption (K3) and (K4), the continuous embedding of
$H^{1}(\Omega)$ into $L^{4}(\Omega)$ (that holds true for $d\leq3$), and the
fact that $u_0 \in \mathcal{A}$ implies $u\in L^{\infty}(0,t;H^{2}(\Omega))$ for all $t>0$
(cf. \cite{MR}), we obtain
\begin{align}
\frac{1}{2}\frac{\mathrm{d}}{\mathrm{d}t}\left\Vert \nabla U\right\Vert _{L^{2}(\Omega)}%
^{2}& +\left\Vert \Delta U\right\Vert _{L^{2}(\Omega)}^{2} \leq\int_{\Omega
}\nabla\left( \mu^{\prime}(u)\nabla w(u)U+\mu(u)\nabla\tilde{w}(U)\right)
(-\Delta U)+g^{\prime}(u)U(-\Delta U)\nonumber\\
& \leq\frac{1}{2}\left\Vert \Delta U\right\Vert _{L^{2}(\Omega)}%
^{2}+C\left\Vert U\right\Vert _{L^{2}(\Omega)}^{2}+C\left\Vert \nabla\left(
\mu^{\prime}(u)\nabla w(u)U+\mu(u)\nabla\tilde{w}(U)\right) \right\Vert
_{L^{2}(\Omega)}^{2}\nonumber\\
& \leq\frac{1}{2}\left\Vert \Delta U\right\Vert _{L^{2}(\Omega)}%
^{2}+C\left\Vert U\right\Vert _{L^{2}(\Omega)}^{2}+C\left\Vert \nabla
U\right\Vert _{L^{2}(\Omega)}^{2}\text{.} \label{cc}%
\end{align}
In particular, we have estimated
\begin{align}
\left\Vert \nabla\left( \mu^{\prime}(u)\nabla w(u)U +\mu(u)\nabla\tilde
{w}(U)\right) \right\Vert _{L^{2}(\Omega)}
& \leq\left\Vert \mu^{\prime\prime}(u)\nabla w(u)U\nabla u\right\Vert
_{L^{2}(\Omega)}
+\left\Vert \mu^{\prime}(u)\Delta w(u)U\right\Vert
_{L^{2}(\Omega)} \nonumber\\
& +\left\Vert \mu^{\prime}(u)\nabla w(u)\nabla U\right\Vert
_{L^{2}(\Omega)} +\left\Vert \mu^{\prime}(u)\nabla u\nabla\tilde{w}(U)\right\Vert
_{L^{2}(\Omega)} \nonumber\\
& +\left\Vert \mu(u)\Delta\tilde{w}(U)\right\Vert _{L^{2}%
(\Omega)},
\nonumber \\
\left\Vert \mu^{\prime\prime}(u)\nabla w(u)U\nabla u\right\Vert _{L^{2}%
(\Omega)} & \leq\left\Vert \mu^{\prime\prime}(u)\nabla w(u)\right\Vert
_{L^{\infty}(\Omega)}\left\Vert U\right\Vert _{L^{4}(\Omega)}\left\Vert \nabla
u\right\Vert _{L^{4}(\Omega)} \nonumber \\
& \leq C\left\Vert U\right\Vert _{H^{1}(\Omega
)}\left\Vert \nabla u\right\Vert _{H^{1}(\Omega)}\leq C\left\Vert U\right\Vert
_{H^{1}(\Omega)}\text{,} \nonumber \\
\vspace{1cm}
\left\Vert \mu^{\prime}(u)\Delta w(u)U\right\Vert _{L^{2}(\Omega)}
& \leq\left\Vert \mu^{\prime}(u)\right\Vert _{L^{\infty}(\Omega)}\left\Vert \Delta
w(u)\right\Vert _{L^{2}(\Omega)}\left\Vert U\right\Vert _{L^{2}(\Omega)} \leq C\left\Vert U\right\Vert _{H^{1}(\Omega)}\text{,} \nonumber \\
\vspace{1cm}
\left\Vert \mu^{\prime}(u)\nabla w(u)\nabla U\right\Vert _{L^{2}(\Omega)} &
\leq\left\Vert \mu^{\prime}(u)\nabla w(u)\right\Vert _{L^{\infty}(\Omega
)}\left\Vert \nabla U\right\Vert _{L^{2}(\Omega)} \leq C\left\Vert 
U\right\Vert _{H^{1}(\Omega)}\text{,} \nonumber \\
\vspace{1cm}
\left\Vert \mu^{\prime}(u)\nabla u\nabla\tilde{w}(U)\right\Vert _{L^{2}%
(\Omega)} & \leq\left\Vert \mu^{\prime}(u)\right\Vert _{L^{\infty}(\Omega
)}\left\Vert \nabla u\right\Vert _{L^{4}(\Omega)}\left\Vert \nabla\tilde
{w}(U)\right\Vert _{L^{4}(\Omega)} \nonumber \\
& \leq C\left\Vert U\right\Vert _{L^{4}%
(\Omega)} \leq C\left\Vert U\right\Vert _{H^{1}(\Omega)}\text{,} \nonumber \\
\vspace{1cm}
\left\Vert \mu(u)\Delta\tilde{w}(U)\right\Vert _{L^{2}(\Omega)} &
\leq\left\Vert \mu(u)\right\Vert _{L^{\infty}(\Omega)}\left\Vert \Delta
\tilde{w}(U)\right\Vert _{L^{2}(\Omega)} \leq C\left\Vert U\right\Vert \nonumber
_{H^{1}(\Omega)}\text{.}%
\end{align}
Note that the above estimates are just formal, however they
can be justified rigorously by mean of an approximation procedure.\\
\noindent
Integrating \eqref{cc} between $s$ and $t$ for some $0<s<t$, we get
\begin{align*}
\left\Vert \nabla U(t)\right\Vert _{L^{2}(\Omega)}^{2} & \leq\left\Vert
\nabla U(s)\right\Vert _{L^{2}(\Omega)}^{2}+C\int_{s}^{t}\left\Vert
U\right\Vert _{L^{2}(\Omega)}^{2}+C\int_{s}^{t}\left\Vert \nabla U\right\Vert
_{L^{2}(\Omega)}^{2}\\
& \leq\left\Vert \nabla U(s)\right\Vert _{L^{2}(\Omega)}^{2}+C\int_{0}%
^{t}\left\Vert U\right\Vert _{L^{2}(\Omega)}^{2}+C\int_{0}^{t}\left\Vert
\nabla U\right\Vert _{L^{2}(\Omega)}^{2}%
\end{align*}
Integrating now between $0$ and $t$ with respect to $s$, and using (\ref{dd}),
we obtain
\begin{align*}
t\left\Vert \nabla U(t)\right\Vert _{L^{2}(\Omega)}^{2} & \leq(1+Ct)\int
_{0}^{t}\left\Vert \nabla U\right\Vert _{L^{2}(\Omega)}^{2}+Ct\int_{0}%
^{t}\left\Vert U\right\Vert _{L^{2}(\Omega)}^{2}\\
& \leq(1+Ct)C\int_{0}^{t}\left\Vert U\right\Vert _{L^{2}(\Omega)}%
^{2}+\left\Vert U(0)\right\Vert _{L^{2}(\Omega)}^{2}\text{.}%
\end{align*}
The right hand side is bounded for every $t>0$ fixed. This concludes the proof
of the lemma.
\end{proof}

\bibliographystyle{plain}
\bibliography{ref_MI}{}

\end{document}